\documentclass[a4paper, 11pt]{amsart} 
\usepackage[
  margin=1.9cm,
  includefoot,
  footskip=30pt,
]{geometry}
\usepackage{amsmath,amssymb,amscd}
\usepackage{amsfonts}
\usepackage[english]{babel}
\usepackage[leqno]{amsmath}
\usepackage{amssymb,amsthm}
\usepackage{mathrsfs} 
\usepackage{amscd}
\usepackage{enumerate}
\usepackage{epsfig}
\usepackage{relsize}
\usepackage{layout}
\usepackage[usenames,dvipsnames]{xcolor}
\usepackage[backref=page]{hyperref}
\usepackage{tikz}
\hypersetup{
 colorlinks,
 citecolor=Green,
 linkcolor=Red,
 urlcolor=Blue}
\usepackage[matrix,arrow,tips,curve]{xy}
\input{xy}
\xyoption{all}
\definecolor{darkgreen}{rgb}{0.0, 0.7, 0.0}
\definecolor{purple}{rgb}{0.5, 0.0, 0.5}
\definecolor{red}{rgb}{0.8, 0.2, 0.0}

\newtheorem{thm}{Theorem}[section]
\newtheorem{bthm}{Theorem}
\newtheorem{bcor}{Corollary}
\newtheorem{lemma}[thm]{Lemma}

\newtheorem{cor}[thm]{Corollary}

\numberwithin{equation}{section}
\theoremstyle{definition}
\newtheorem{defi}[thm]{Definition}

\newtheorem{notation}[thm]{Notation}

\theoremstyle{remark}
\newtheorem{remark}[thm]{Remark}

\newcommand{\Z}{\mathbb{Z}}
\newcommand{\C}{\mathbb{C}}
\newcommand{\Q}{\mathbb{Q}}

\newcommand{\Pic}{\operatorname{Pic}}
\newcommand{\Hilb}{\operatorname{Hilb}}

\DeclareMathOperator{\Hom}{{Hom}}
\DeclareMathOperator{\Ext}{{Ext}}

\def \Im{{\rm Im}}

\DeclareMathOperator{\id}{id}

\def \P{\mathbb{P}}

\def \F{\mathcal F}

\def\I{{\mathcal J}}
\def \L{\mathcal L}
\def \E{\mathcal E}
\def \G{\mathcal G}
\def \H{\mathcal H}

\def\O{\mathcal O}

\def\M0{\mathcal M^0}

\newcommand{\SHom}{{\mathcal{H}om}}
\newcommand{\SExt}{{\mathcal{E}xt}}

\DeclareMathOperator{\Sing}{{Sing}}

\DeclareMathOperator{\Ker}{{Ker}}

\newcommand{\Num}{\operatorname{Num}}
\begin{document}

\title[Ulrich subvarieties and low rank Ulrich bundles on complete intersections]{Ulrich subvarieties and the non-existence of low rank Ulrich bundles on complete intersections}

\author[A.F. Lopez, D. Raychaudhury]{Angelo Felice Lopez and Debaditya Raychaudhury}

\address{\hskip -.43cm Angelo Felice Lopez, Dipartimento di Matematica e Fisica, Universit\`a di Roma
Tre, Largo San Leonardo Murialdo 1, 00146, Roma, Italy. e-mail {\tt angelo.lopez@uniroma3.it}}

\address{\hskip -.43cm Debaditya Raychaudhury, Department of Mathematics, University of Arizona, 617 N Santa Rita Ave., Tucson, AZ 85721, USA. email: {\tt draychaudhury@math.arizona.edu}}

\thanks{The first author was partially supported by the GNSAGA group of INdAM and by the PRIN ``Advances in Moduli Theory and Birational Classification''.}


\begin{abstract} 
We characterize the existence of an Ulrich vector bundle on a variety $X \subset \P^N$ in terms of the existence of a subvariety satisfying some precise conditions. Then we use this fact to prove that a complete intersection of dimension $n \ge 4$, which if $n=4$ is very general and not of type $(2,2)$, does not carry any Ulrich bundles of rank $r \le 3$ unless $n=4, r=2$ and $X$ is a quadric.
\end{abstract}

\maketitle

\section{Introduction}

It is a well-known principle, in algebraic geometry, that the geometry of a given variety $X$ is often governed by its subvarieties. In a similar fashion, also vector bundles on $X$ give important information on its geometry. In many cases these two aspects have met, giving rise to deeper understanding. A celebrated example of this is the Hartshorne-Serre correspondence. 

It is the first purpose of this paper to highlight another instance of the above in the case of Ulrich vector bundles. 

Let $X \subset \P^N$ is a smooth irreducible variety of dimension $n \ge 1$. A vector bundle $\E$ on $X$ is called Ulrich if $H^i(\E(-p))=0$ for all $i \ge 0$ and $1 \le p \le n$. While the importance of Ulrich vector bundles is well-known (see for example \cite{es, b2, cmp} and references therein), the main general problem about them is their conjectural existence. 

With this in mind, the starting point of this research was to study which subvarieties of $X$ one can associate to an Ulrich bundle. This is of course not a new idea, as it has been proposed several times by many authors (in fact already for aCM bundles), see for example \cite{ak, b1, b2, c, ch, ckl, hh} and references therein. A more systematic approach was recently given in \cite{cfk}, where the existence of an Ulrich bundle on a threefold $X$ was related to the existence of a curve on $X$ satisfying some properties. 

Our first task has been to generalize, essentially in the same way, the above result \cite[Thm.~3.1]{cfk} to any $n$-dimensional variety. In order to state it, we need to introduce some notation. 

Let $Z \subset X$ be a Cohen-Macaulay, pure codimension $2$ subvariety and let $D$ be a divisor on $X$. The short exact sequence
$$0 \to  \I_{Z/X}(K_X+D) \to \O_X(K_X+D) \to \O_Z(K_X+D) \to 0$$
determines a coboundary map
$$\gamma_{Z,D} : H^{n-2}(\O_Z(K_X+D)) \to H^{n-1}(\I_{Z/X}(K_X+D))$$
whose dual, by Serre duality, is
\begin{equation}
\label{gammastar}
\gamma_{Z,D}^*: \Ext^1_{\O_X}(\I_{Z/X}(D), \O_X) \to H^0(\omega_Z(-K_X-D)).
\end{equation}
Moreover, assume that either $Z$ is as above or is empty. For any subspace $W \subseteq \Ext^1_{\O_X}(\I_{Z/X}(D), \O_X)$ and for any line bundle $\L$ on $X$, one can define a natural map (see \eqref{delta})
$$\delta_{Z,W,\L} : H^{n-1}(\I_{Z/X}(D) \otimes \L) \to W^* \otimes H^n(\L).$$
Then we have

\eject

\begin{bthm} 
\label{1-1}
Let $X \subset \P^N$ be a smooth irreducible variety of dimension $n \ge 2$, degree $d \ge 2$ and let $D$ be a divisor on $X$. Then $(X,\O_X(1))$ carries a rank $r \ge 2$ Ulrich vector bundle $\E$ with $\det \E = \O_X(D)$ if and only if there is a subvariety $Z \subset X$ such that:
\begin{itemize}
\item[(a)] $Z$ is either empty or of pure codimension $2$,
\item[(b)] if $Z \ne \emptyset$ and either $r=2$ or $n \le 5$, then $Z$ is smooth (possibly disconnected),
\item[(c)] if $Z \ne \emptyset$ and $n \ge 6$, then $Z$ is either smooth or is normal, Cohen-Macaulay, reduced and with $\dim \Sing(Z) = n-6$,
\end{itemize}
and there is a $(r-1)$-dimensional subspace $W \subseteq \Ext^1_{\O_X}(\I_{Z/X}(D), \O_X)$ such that the following hold:
\begin{itemize}
\item[(i)] If $Z \ne \emptyset$, then $\gamma_{Z,D}^*(W)$ generates $\omega_Z(-K_X-D)$ (that is the multiplication map 

\noindent $\gamma_{Z,D}^*(W) \otimes \O_Z \to \omega_Z(-K_X-D)$ is surjective).
\item[(ii)] $H^0(K_X+nH-D)=0$.
\item[(iii)] $H^0(\I_{Z/X}(D-H))=0$.
\item[(iv)] If $n \ge 3$, then $H^i(\I_{Z/X}(D-pH))=0$ for $1 \le i \le n-2, 1 \le p \le n$.
\item[(v)] $(-1)^{n-1}\chi(\I_{Z/X}(D-pH))=(r-1) \chi(K_X+pH)$, for $1 \le p \le n$.
\item[(vi)] $\delta_{Z, W, -nH} : H^{n-1}(\I_{Z/X}(D-nH)) \to W^* \otimes H^n(-nH)$ is either injective or surjective.
\end{itemize}
Moreover the following exact sequences hold
$$0 \to W^* \otimes \O_X \to \E \to \I_{Z/X}(D) \to 0$$
and, if $Z \ne \emptyset$,
$$0 \to \O_X(-D) \to \E^* \to  W \otimes \O_X \to \omega_Z(-K_X-D) \to 0.$$
\end{bthm}

With this result at hand, one can start exploring the existence problem for Ulrich bundles in geometric terms, by using subvarieties. We thus give the following

\begin{defi} 
\label{ulrsub}
Let $r \ge 2$ and let $X \subset \P^N$ be a smooth irreducible variety of dimension $n \ge 2$, degree $d \ge 2$ and let $D$ be a divisor on $X$. An {\it Ulrich subvariety of $X$} is a subvariety $Z \subset X$ carrying a $(r-1)$-dimensional subspace $W \subseteq \Ext^1_{\O_X}(\I_{Z/X}(D), \O_X)$ such that properties (a)-(c) and (i)-(vi) of Theorem \ref{1-1} hold.
\end{defi}

A priori, an Ulrich subvariety can be empty. In that case the conditions of the theorem hold with $\I_{Z/X}=\O_X$. On the other hand, several simple hypotheses can be given to check that $Z$ is nonempty, irreducible and to apply the conditions in Theorem \ref{1-1}, see Remarks \ref{primo} and \ref{Znonv}.

Now, Theorem \ref{1-1} says that a variety $X$ carries an Ulrich bundle if and only if $X$ contains an Ulrich subvariety. Hence the question becomes: when do Ulrich subvarieties exist on a given $X$? 

According to the references given before, several examples can be given. To give an explicit, probably well-known example (see, e.g. \cite[Prop.~8.2]{b1}, \cite[Thm.~2.1]{cf}), consider a smooth hypersurface $X \subset \P^{n+1}$ of degree $d \ge 2$ and dimension $n \ge 2$. Then we prove in Corollary \ref{hyper2} that, in rank $2$ (if $n=2$ we also need to assume that $\det \E = \O_X(d-1)$), an Ulrich subvariety of $X$ is a smooth $(n-2)$-dimensional arithmetically Gorenstein subvariety $Z \subset X$, irreducible when $n \ge 3$, with minimal free resolution
$$0 \to \O_{\P^{n+1}}(-2d+1) \to \O_{\P^{n+1}}(-d)^{\oplus (2d-1)} \to \O_{\P^{n+1}}(-d+1)^{\oplus (2d-1)} \to \I_{Z/\P^{n+1}} \to 0.$$
As a consequence, for example if $n \ge 5$, one sees very quickly that an Ulrich subvariety cannot be contained in a smooth hypersurface of degree $d$. More generally, we recover in the case of Ulrich bundles, in a unified way, several facts known for aCM bundles (see \cite[Prop.~7.6(b)]{b1}, \cite[Thm.~1.3]{cm}, \cite[Main Thm.]{krr1}), \cite[Thm.~1.1(2)(b)]{krr2}, \cite[Prop.~3.2]{kl}, \cite[Thm.~1.1(1)]{krr2})(as a matter of fact, by \cite[Prop.~7.6(b)]{b1}, (i) below holds also for $X$ general). %(as a matter of fact, by the above results, (i) below holds also for $d=16$ and (iii) also for $d=3$. The latter is also shown for very general in Theorem \ref{ci}). 
\begin{bcor} 
\label{hyper3}

Let $n \ge 2$ and let $X \subset \P^{n+1}$ be a smooth hypersurface of degree $d \ge 2$. Then $X$ does not carry rank $2$ Ulrich vector bundles if any of the following holds:
\begin{itemize}
\item[(i)] $X$ is very general, $n=2$ and $d \ge 16$. 
\item[(ii)] $X$ is general, $n=3$ and $d \ge 6$.
\item[(iii)] $X$ is general, $n=4$ and $d \ge 3$.
\item[(iv)] $n \ge 5$.
\end{itemize}
\end{bcor}
Previously, these results have been proved, as far as we know, using the crucial fact that, in rank $2$, if a hypersurface $X$ carries an Ulrich subvariety, then it is Pfaffian. But as soon as one considers other simple varieties, such as complete intersections, or higher rank, this algebraic property is not available any more and the question of existence of low rank Ulrich bundles has remained open, so far (except when $X$ is a hypersurface and $r=3$ \cite{rt1, rt2, tr}, or a complete intersection of degrees $d_i \gg 0$ and $r=2$ \cite{br}, or a complete intersection of low degree \cite{c, fi} and $r=2$; see also \cite{hh, er, m1, m2}).

Using Theorem \ref{1-1}, we show that, in low-rank, there are no Ulrich subvarieties on complete intersections. In fact, we have

\begin{bthm} 
\label{ci}

Let $c \ge 1, n \ge 4$ and let $X \subset \P^{n+c}$ be a smooth complete intersection of hypersurfaces of degrees $(d_1, \ldots, d_c)$ with $d_i \ge 2, 1 \le i \le c$. Assume that one of the following holds:
\begin{itemize}
\item[(a)] $n \ge 5$, or 
\item[(b)] $n=4$, $X$ is very general and is not of type $(2,2)$.
\end{itemize}
Then $X$ does not carry Ulrich vector bundles of rank $r \le 3$, unless $n=4, r=2$ and $X \subset \P^5$ is a quadric. 
\end{bthm}
Note that, in the case $n=4$, $X$ of type $(2,2)$ of Theorem \ref{ci}, Ulrich bundles of rank $2$ do exist by \cite[Thm.~5.5]{es2}. 

In a forthcoming paper \cite{lr3} we will apply Theorem \ref{1-1} to study low rank Ulrich bundles on Veronese varieties.

\section{Notation and conventions}

Throughout the paper we work over the complex numbers. 

We henceforth establish the following 

\begin{notation} 

\hskip 3cm

\begin{itemize} 
\item $X \subset \P^N$ is a smooth irreducible variety of dimension $n \ge 1$.
\item $H \in |\O_X(1)|$ is a very ample divisor. 
\item $d=H^n$ is the degree of $X$.
\item If $Y \subseteq \P^N$ is a closed subscheme, $I_Y$ is its saturated homogeneous ideal. 
\item We say that $X$ is {\it subcanonical} if $-K_X=i_XH$, for some $i_X \in \Z$.
\item We write {\it aCM} for {\it arithmetically Cohen-Macaulay} and {\it aG} for {\it arithmetically Gorenstein}.
\item We use the convention $\binom{\ell}{m}=\frac{\ell (\ell-1)\ldots (\ell-m+1)}{m!} \ \mbox{for} \ m \ge 1, \ell \in \Z$.
Note that $\binom{-\ell}{m}=(-1)^m\binom{\ell+m-1}{m}$ and $\chi(\O_{\P^m}(\ell))=\binom{\ell+m}{m}$.
\end{itemize} 
\end{notation}

\section{Generalities on Ulrich vector bundles}

We will often use the following well-known properties of Ulrich bundles.

\begin{lemma}
\label{ulr}
Let $X \subset \P^N$  and let $\E$ be a rank $r$ Ulrich bundle. We have:
\begin{itemize}
\item[(i)] $\E$ is globally generated. 
\item[(ii)] $\E$ is aCM.
\item[(iii)] $c_1(\E) H^{n-1}=\frac{r}{2}[K_X+(n+1)H] H^{n-1}$.
\item [(iv)] $\E_{|Y}$ is Ulrich on a smooth hyperplane section $Y$ of $X$.
\item[(v)] $\det \E$ is globally generated and it is non trivial, unless $(X, H, \E) = (\P^n, \O_{\P^n}(1), \O_{\P^n}^{\oplus r})$.
\item[(vi)] $H^0(\E^*)=0$, unless $(X, H, \E) = (\P^n, \O_{\P^n}(1), \O_{\P^n}^{\oplus r})$.
\item [(vii)] $\O_X(l)$ is Ulrich if and only if $(X,H,l)=(\P^n,\O_{\P^n}(1),0)$.
\item[(viii)] If $n \ge 2$, then $c_2(\E) H^{n-2}=\frac{1}{2}[c_1(\E)^2-c_1(\E) K_X] H^{n-2}+\frac{r}{12}[K_X^2+c_2(X)-\frac{3n^2+5n+2}{2}H^2] H^{n-2}$.
\item[(ix)] $\chi(\E(m))= rd\binom{m+n}{n}$.
\end{itemize} 
\end{lemma}
\begin{proof}
See for example \cite[Lemma 3.2]{lr} for (i)-(iv) and (viii), (ix), \cite[Lemma 2.1]{lo} for (v), \cite[Lemma 2.1]{cfk} for (vi) and \cite[Lemma 4.2]{aclr} for (vii). 
\end{proof}

The following construction will be fundamental in the paper. 

\begin{lemma}
\label{zeta}
Let $n \ge 2$, let $X \subset \P^N$ and let $\E$ be a rank $r \ge 2$ Ulrich bundle with $\det \E = \O_X(D)$. Then there is a subvariety $Z \subset X$ such that:
\begin{itemize}
\item[(a)] $Z$ is either empty or of pure codimension $2$.
\item[(b)] If $Z \ne \emptyset$, then $[Z]=c_2(\E)$.
\item[(c)] If $Z \ne \emptyset$ and either $r=2$ or $n \le 5$, then $Z$ is smooth (possibly disconnected).
\item[(d)] If $Z \ne \emptyset$ and $n \ge 6$, then $Z$ is either smooth or is normal, Cohen-Macaulay, reduced and with $\dim \Sing(Z) = n-6$. Moreover, if $\Sing(Z) \ne \emptyset$, then $[\Sing(Z)]=c_3(\E)^2-c_2(\E)c_4(\E) \in A^6(X)$.
\item[(e)] There is an effective Cartier divisor $Y \in |\det(\E)|$, possibly empty, such that $Z \subset Y, \Sing(Z) \subseteq \Sing(Y) \cap Z$. Also, if $Y \ne \emptyset$, then $Y$ is smooth if $n \le 3$, while, if $n \ge 4$, then $Y$ is either smooth or  $\dim \Sing(Y) = n-4$ and $[\Sing(Y)]=c_2(\E)^2-c_1(\E)c_3(\E) \in A^4(X)$.
\end{itemize}
Also, the following hold:
\begin{itemize}
\item[(i)] If $Z \ne \emptyset$, then $\deg(Z)=\frac{1}{2}D^2 H^{n-2}-\frac{1}{2}D K_X H^{n-2}-\frac{rd}{24}(3n^2+5n+2)+\frac{r}{12}K_X^2 H^{n-2}+\frac{r}{12}c_2(X) H^{n-2}$.
\item[(ii)] There is an exact sequence
\begin{equation}
\label{est-1}
0 \to \O_X^{\oplus (r-1)} \to \E \to \I_{Z/X}(D) \to 0
\end{equation}
(with the convention that $\I_{Z/X}=\O_X$ if $Z=\emptyset$).
\item[(iii)] If $D=uH$, we have for any $m \in \Z$:
$$\chi(\O_Z(m))=\binom{m+N}{N}-rd\binom{m-u+n}{n}+(r-1)\binom{m-u+N}{N}-\chi(\I_{X/\P^N}(m))-(r-1)\chi(\I_{X/\P^N}(m-u))$$
(with the convention that $\chi(\O_Z(m))=0$ if $Z=\emptyset$).
\end{itemize}
Moreover assume that $Z$ is nonempty and smooth. Then there exists a rank $r-2$ vector bundle $\F_Z$ on $Z$ sitting in the exact sequence
\begin{equation}
\label{nor1}
0 \to \F_Z \to \E(-D)_{|Z} \to N_{Z/X}^* \to 0
\end{equation}
and satisfying:
\begin{itemize}
\item[(iv)] If $r=2$, then $\F_Z=0$, that is $N_{Z/X} \cong \E_{|Z}$ and $\omega_Z \cong \O_Z(K_X+D)$.
\item[(v)] If $r = 3$, then $\F_Z=\omega_Z(-K_X-2D)$.
\item[(vi)] If $r \ge 4$, then $\det(\F_Z)= \omega_Z(-K_X-(r-1)D)$ and we have an exact sequence
\begin{equation}
\label{nor2}
0 \to \omega_Z^{-1}(K_X) \to \O_Z(-D)^{\oplus (r-1)} \to \F_Z \to 0.
\end{equation}
\item[(vii)] If $r=2$ or $r \ge 4$, then $c_2(Z)=c_2(X)_{|Z}-c_2(\E)_{|Z}+K_Z^2-K_Z {K_X}_{|Z}$.
\item[(viii)] If $r=3$, then 

$c_2(Z)=c_2(X)_{|Z}-c_2(\E)_{|Z}-c_1(\E)^2_{|Z}+K_Z {K_X}_{|Z}-{K_X^2}_{|Z}+2K_Z c_1(\E)_{|Z}-2{K_X}_{|Z} c_1(\E)_{|Z}$.
\end{itemize}
\end{lemma}
\begin{proof} 
By Lemma \ref{ulr}(i) we can choose a general subspace $V \subset H^0(\E)$ such that $\dim V = r-1$, thus giving rise to a general morphism $\varphi : V \otimes \O_X  \to \E$. It is well known (see for example \cite[Statement (folklore)(i), \S 4.1]{ba} and \cite[Ch.~VI, \S 4, page 257]{acgh}) that there is an exact sequence
$$0 \to V \otimes \O_X \to \E \to \I_{Z/X}(D) \to 0$$
where $Z = D_{r-2}(\varphi)$ is the degeneracy locus, so that $Z$ is either empty or is of pure codimension $2$ and in the latter case $[Z]=c_2(\E)$. Moreover, $\Sing(Z)=D_{r-3}(\varphi)$ that is either empty or of codimension $6$. Hence $Z$ is smooth if $r=2$ or if $n \le 5$, while, if $\Sing(Z) \ne \emptyset$, then $\dim \Sing(Z) = n-6$ when $n \ge 6$ and in that case $Z$ is normal, Cohen-Macaulay and reduced (see for example \cite[Prop.~2.4]{tt}). Also, if $\Sing(Z) \ne \emptyset$, then $[\Sing(Z)]=c_3(\E)^2-c_2(\E)c_4(\E)$ by Porteous' formula (see for example \cite[Thm.~12.4]{eh}). This proves (a)-(d) and (ii). To see (e) choose a general subspace $V' \subset H^0(\E)$ such that $\dim V' = r$ and $V \subset V'$. Then we get a general morphism $\varphi' : V \otimes \O_X  \to \E$ and setting $Y=D_{r-1}(\varphi')$ we see that $Y \in |\det(\E)|$ and $Z \subset Y$. As above, we have that $\Sing(Y)=D_{r-2}(\varphi')$, thus it has the required properties. This proves (e). As for (i), since $[Z]=c_2(\E)$, we have that $\deg(Z)=c_2(\E) H^{n-2}$ is given by Lemma \ref{ulr}(viii). To see (iii), we use the exact sequences
$$0 \to \I_{Z/X}(m) \to \O_X(m) \to \O_Z(m) \to 0, \ \ \ 0 \to \I_{X/\P^N}(m) \to \O_{\P^N}(m) \to \O_X(m) \to 0$$
and the one obtained from \eqref{est-1}
$$0 \to \O_X^{\oplus (r-1)}(m-u) \to \E(m-u) \to \I_{Z/X}(m) \to 0.$$
We get, using Lemma \ref{ulr}(ix), 
\begin{equation*}
\begin{aligned}
& \chi(\O_Z(m)) = \\
& = \chi(\O_X(m))-\chi(\I_{Z/X}(m))=\chi(\O_{\P^N}(m))-\chi(\I_{X/\P^N}(m))-\chi(\E(m-u))+(r-1)\chi(\O_X(m-u))= \\
& = \chi(\O_{\P^N}(m))-\chi(\I_{X/\P^N}(m))-\chi(\E(m-u))+ (r-1)\chi(\O_{\P^N}(m-u))-(r-1)\chi(\I_{X/\P^N}(m-u))=\\ 
& =\binom{m+N}{N}-rd\binom{m-u+n}{n}+(r-1)\binom{m-u+N}{N}-\chi(\I_{X/\P^N}(m))-(r-1)\chi(\I_{X/\P^N}(m-u)).
\end{aligned}
\end{equation*}
This proves (iii). To see (iv)-(vi), we use that $\I_{Z/X} \otimes \O_Z \cong N_{Z/X}^*$. By \eqref{est-1} we find the exact sequence 
$$0 \to \O_X(-D)^{\oplus (r-1)} \to \E(-D) \to \I_{Z/X} \to 0$$
so that, tensoring by $\O_Z$, and we have another exact sequence
$$\O_Z(-D)^{\oplus (r-1)} \to \E(-D)_{|Z} \to N_{Z/X}^* \to 0.$$
Setting $\F_Z=\Ker\{\E(-D)_{|Z} \to N_{Z/X}^*\}$, we get \eqref{nor1} and setting $\G=\Ker\{\O_Z(-D)^{\oplus (r-1)} \to \F_Z\}$ we also have the exact sequence
$$0 \to \G \to \O_Z(-D)^{\oplus (r-1)} \to \F_Z \to 0.$$
Since $\E(-D)_{|Z}, N_{Z/X}^*$ and $\O_Z(-D)^{\oplus (r-1)}$ are all locally free, we see that the same holds for $\F_Z$ and $\G$. Hence $\F_Z$ has rank $r-2$ and $\G$ is a line bundle. This proves (iv) and also (v)-(vi) by taking determinants. Finally (vii) and (viii) follow computing Chern classes from \eqref{nor1}, \eqref{nor2} and the exact sequence
$$0 \to T_Z \to {T_X}_{|Z} \to N_{Z/X} \to 0.$$
\end{proof}
Next, we study the ideal of $Z$ as in Lemma \ref{zeta}. Even though this lemma will not be used in the sequel, it gives some properties of the ideal of an Ulrich subvariety that could be useful in future papers.
\begin{lemma}
\label{genide}

Let $n \ge 2, d \ge 2$, let $X \subset \P^N$ and let $\E$ be a rank $r \ge 2$ Ulrich vector bundle on $X$ with $\det \E= \O_X(u)$. Let $Z$ be a nonempty subvariety associated to $\E$ as in Lemma \ref{zeta}. If
\begin{itemize}
\item[(i)] $\I_{X/\P^N}(u)$ is globally generated, and
\item[(ii)] $H^1(\I_{X/\P^N}(u))=0$,
\end{itemize} 
then $\I_{Z/\P^N}(u)$ is globally generated. If
\begin{itemize}
\item[(iii)] $I_X$ is generated in degree $u$,
\item[(iv)] $H^1(\O_X(l))=0$ for $l \ge 1$, and
\item[(v)] $X$ is projectively normal,
\end{itemize}
then $I_Z$ is generated in degree $u$ and not in lower degree.
\end{lemma}
\begin{proof}
Since $\E$ is $0$-regular, by Castelnuovo-Mumford \cite[Thm.~1.8.5]{laz1} we have that 
\begin{equation}
\label{cm}
H^0(\E) \otimes H^0(\O_X(l)) \to H^0(\E(l)) \ \hbox{is surjective for every} \ l \ge 0.
\end{equation}
Assume (i)-(ii). The exact sequence
\begin{equation}
\label{idea3}
0 \to \I_{X/\P^N}(u) \to \I_{Z/\P^N}(u) \to \I_{Z/X}(u) \to 0
\end{equation}
is exact in global sections by (ii). From \eqref{est-1} and Lemma \ref{ulr}(i) it follows that 
$\I_{Z/X}(u)$ is globally generated. Hence we also have a surjective morphism 
$$H^0(\I_{Z/X}(u)) \otimes \O_{\P^N} \twoheadrightarrow \I_{Z/X}(u).$$
Therefore we get by (i) a commutative diagram
$$\xymatrix{0 \ar[r] & H^0(\I_{X/\P^N}(u)) \otimes \O_{\P^N} \ar[r] \ar@{->>}[d] & H^0(\I_{Z/\P^N}(u)) \otimes \O_{\P^N} \ar[r] \ar[d]^f & H^0(\I_{Z/X}(u)) \otimes \O_{\P^N} \ar[r] \ar@{->>}[d] & 0 \\ 0 \ar[r] & \I_{X/\P^N}(u) \ar[r] & \I_{Z/\P^N}(u) \ar[r] & \I_{Z/X}(u) \ar[r] & 0}$$ showing that $f$ is surjective, hence that $\I_{Z/\P^N}(u)$ is globally generated. 

Now assume (iii)-(v). Fix $l \in \Z$ with $l \ge 1$. The exact sequence
$$0 \to \I_{X/\P^N}(l) \to \O_{\P^N}(l) \to \O_X(l) \to 0$$
and (iv) show that 
\begin{equation}
\label{ox}
H^2(\I_{X/\P^N}(l))=0.
\end{equation}
Also, (iv) and \eqref{est-1} imply that we have a surjection $H^0(\E(l)) \to H^0(\I_{Z/X}(u+l))$. By \eqref{cm} and the commutative diagram 
$$\xymatrix{H^0(\E) \otimes H^0(\O_X(l)) \ar@{->>}[r] \ar[d] & H^0(\E(l)) \ar@{->>}[d] \\ H^0(\I_{Z/X}(u)) \otimes H^0(\O_X(l)) \ar[r] & H^0(\I_{Z/X}(u+l))}$$
we deduce that the map $H^0(\I_{Z/X}(u)) \otimes H^0(\O_X(l)) \to H^0(\I_{Z/X}(u+l))$ is surjective. Thus, by (v), so does the map
$H^0(\I_{Z/X}(u)) \otimes H^0(\O_{\P^N}(l)) \to H^0(\I_{Z/X}(u+l))$. Thus, in the commutative diagram
$$\xymatrix{& 0 \ar[d] & 0 \ar[d] \\ & H^0(\I_{X/\P^N}(u)) \otimes H^0(\O_{\P^N}(l))  \ar[r]^{\hskip .9cm f'} \ar[d] & H^0(\I_{X/\P^N}(u+l))  \ar[d] \\ & H^0(\I_{Z/\P^N}(u)) \otimes H^0(\O_{\P^N}(l)) \ar[r]^{\hskip .9cm f} \ar[d] & H^0(\I_{Z/\P^N}(u+l)) \ar[d] \\ & H^0(\I_{Z/X}(u)) \otimes H^0(\O_{\P^N}(l)) \ar@{->>}[r] \ar[d] & H^0(\I_{Z/X}(u+l)) \ar[d] \\ & 0  & 0}$$
we have that $f'$ is surjective by (iii), hence so does $f$. Therefore $I_Z$ is generated in degree $u$. Finally,  tensoring \eqref{est-1} by $\O_X(-1)$, we get that $H^0(\I_{Z/X}(u-1))=0$, hence tensoring \eqref{idea3} by $\O_{\P^N}(-1)$ shows that $H^0(\I_{X/\P^N}(u-1))=H^0(\I_{Z/\P^N}(u-1))$, that is every hypersurface of degree $u-1$ that contains $Z$ must contain $X$. It follows that $I_Z$ is not generated in degree $u-1$, for otherwise $Z$ would be cut out scheme-theoretically by hypersurfaces of degree $u-1$. 
\end{proof}

\section{Ulrich of rank $r \ge 2$ and codimension $2$ subvarieties}
\label{corr}

In this section we will prove Theorem \ref{1-1}. We will then give several remarks that help and simplify its applications.

Let $Z \subset X$ be a Cohen-Macaulay, pure codimension $2$ subvariety and let $D$ be a divisor on $X$. The short exact sequence
$$0 \to  \I_{Z/X}(K_X+D) \to \O_X(K_X+D) \to \O_Z(K_X+D) \to 0$$
determines a coboundary map
$$\gamma_{Z,D} : H^{n-2}(\O_Z(K_X+D)) \to H^{n-1}(\I_{Z/X}(K_X+D))$$
whose dual, by Serre duality, is
$$\gamma_{Z,D}^*: \Ext^1_{\O_X}(\I_{Z/X}(D), \O_X) \to H^0(\omega_Z(-K_X-D)).$$ 
Moreover, assume that either $Z$ is as above or is empty, in which case $\I_{Z/X}=\O_X$. For any subspace $W \subseteq \Ext^1_{\O_X}(\I_{Z/X}(D), \O_X) \cong H^{n-1}(\I_{Z/X}(K_X+D))^*$, we obtain a surjection
$$\phi_W : H^{n-1}(\I_{Z/X}(K_X+D)) \twoheadrightarrow W^*.$$ 
Since, by Serre duality, 
$$\Ext^1_{\O_X}(\I_{Z/X}(D), W^* \otimes \O_X) \cong \Hom(H^{n-1}(\I_{Z/X}(K_X+D)), W^*)$$
we get an extension
\begin{equation}
\label{est}
0 \to W^* \otimes \O_X \to \E \to \I_{Z/X}(D) \to 0
\end{equation}
associated to $\phi_W$. This also allows to define, for any line bundle $\L$ on $X$, the map
\begin{equation}
\label{delta}
\delta_{Z,W,\L} : H^{n-1}(\I_{Z/X}(D) \otimes \L) \to W^* \otimes H^n(\L)
\end{equation}
so that, in particular, $\delta_{Z,W,K_X}=\phi_W$.

Note that $\SHom_{\O_X}(\I_{Z/X}, \O_X) \cong \O_X$ by \cite[Lemma IV.5.1]{alkl}, hence 
$$\SHom_{\O_X}(\I_{Z/X}(D), \O_X) \cong \O_X(-D).$$ 
Also, when $Z \ne \emptyset$, we have that 
$$\omega_Z = \SExt^2_{\O_X}(\O_Z, \omega_X) \cong \SExt^1_{\O_X}(\I_{Z/X}, \omega_X) \cong \SExt^1_{\O_X}(\I_{Z/X}, \O_X)(K_X)$$ 
so that  $\SExt^1_{\O_X}(\I_{Z/X}(D)), \O_X) \cong \omega_Z(-K_X-D)$. Thus, dualizing \eqref{est}, we get an exact sequence
\begin{equation}
\label{est20}
0 \to \O_X(-D) \to \E^* \to  W \otimes \O_X \to \omega_Z(-K_X-D) \to \SExt^1_{\O_X}(\E,\O_X) \to 0
\end{equation}
and a commutative diagram
\begin{equation}
\label{diag20}
\xymatrix{& W \otimes \O_X \ar@{->>}[d] \ar[r] & \omega_Z(-K_X-D) \\ & \gamma_{Z,D}^*(W) \otimes \O_Z \ar[ur] & }
\end{equation}
that shows, in particular, that 
\begin{equation}
\label{im}
\Im \{W \to H^0(\omega_Z(-K_X-D))\} = \gamma_{Z,D}^*(W).
\end{equation}
Also, if $\E$ is locally free, then \eqref{est20} becomes
\begin{equation}
\label{est2}
0 \to \O_X(-D) \to \E^* \to  W \otimes \O_X \to \omega_Z(-K_X-D) \to 0.
\end{equation}

Moreover, consider the map 
\begin{equation}
\label{alfa}
\alpha:=\gamma_{Z,D-K_X-nH}: H^{n-2}(\O_Z(D-nH))  \to H^{n-1}(\I_{Z/X}(D-nH))
\end{equation} 
and the multiplication map 
\begin{equation}
\label{mu}
\mu_{Z, W} : \gamma_{Z,D}^*(W) \otimes H^0(K_X+nH) \to H^0(\omega_Z(nH-D)).
\end{equation} 
Setting $\delta=\delta_{Z,W,-nH}$, we have a commutative diagram
$$\xymatrix{& H^{n-1}(\I_{Z/X}(D-nH)) \ar[r]^{\hskip .4cm \delta} & W^* \otimes H^n(-nH) \\ & H^{n-2}(\O_Z(D-nH)) \ar[u]^{\alpha} \ar[ur] & }$$
that dualizes to
\begin{equation}
\label{diam}
\xymatrix{& W \otimes H^0(K_X+nH) \ar@{->>}[r]^{\hskip -.6cm \gamma_{Z,D}^* \otimes \id} \ar[d]^{\delta^*} & \gamma_{Z,D}^*(W) \otimes H^0(K_X+nH) \ar[r]^{\hskip .6cm \mu_{Z, W}} & H^0(\omega_Z(nH-D)) \\ & H^{n-1}(\I_{Z/X}(D-nH))^* \ar[urr]^{\alpha^*}}.
\end{equation}
Now we prove our first main result.
 
\renewcommand{\proofname}{Proof of Theorem \ref{1-1}}
\begin{proof}
Let $\E$ be a rank $r$ Ulrich vector bundle on $X$ with $\det \E = \O_X(D)$. 

Let $Z$ be a subvariety associated to $\E$ as in Lemma \ref{zeta}. Then $Z$ satisfies (a)-(c) by Lemma \ref{zeta}. Also, as in the proof of Lemma \ref{zeta}, we have an exact sequence
$$0 \to V \otimes \O_X \to \E \to \I_{Z/X}(D) \to 0$$
so that, setting $W = V^*$, we get \eqref{est} and, when $Z \ne \emptyset$, dualizing it gives \eqref{est2}. Since $H^n(\E(K_X))=H^0(\E^*)=0$ by Lemma \ref{ulr}(vi), we get by \eqref{est} a surjection 
$$H^{n-1}(\I_{Z/X}(K_X+D)) \twoheadrightarrow H^n(W^* \otimes \O_X(K_X)) \cong W^*$$
and therefore an inclusion $W \subseteq H^{n-1}(\I_{Z/X}(K_X+D))^* \cong \Ext^1_{\O_X}(\I_{Z/X}(D), \O_X)$. 

We now show that (i)-(vi) hold. 

First, (i) follows by \eqref{diag20} and \eqref{est2}. Next, \eqref{est}, the fact that $\E$ is Ulrich and Kodaira vanishing  imply that 
\begin{equation}
\label{va}
H^i(\I_{Z/X}(D-pH))=0 \ \hbox{for} \ i \in \{0, \ldots n-2, n\}, 1 \le p \le n.
\end{equation}
and
\begin{equation}
\label{va1}
\delta_{Z, W, -pH} : H^{n-1}(\I_{Z/X}(D-pH)) \to W^* \otimes H^n(\O_X(-pH)) \ \hbox{is an isomorphism for} \ 1 \le p \le n.
\end{equation}
Now \eqref{va} gives (iii) and (iv), while, together with the exact sequence
$$0 \to \I_{Z/X}(D-nH) \to \O_X(D-nH) \to \O_Z(D-nH) \to 0$$
we get that $H^n(\O_X(D-nH)=0$, that is (ii) by Serre duality. Moreover  \eqref{va}, \eqref{va1} and Kodaira vanishing imply (v), while \eqref{va1} for $p=n$ gives (vi).

Conversely, assume that we are given a subvariety $Z \subset X$ as in the statement of the theorem and a $(r-1)$-dimensional subspace $W \subseteq \Ext^1_{\O_X}(\I_{Z/X}(D), \O_X)$, such that the conditions (i)-(vi) hold. Then, as explained before the theorem, there is a sheaf $\E$ arising as the extension \eqref{est}. Our first aim is to prove that $\E$ is locally free. As is well known, this will follow from 
\begin{equation}
\label{vb}
\SExt^i_{\O_X}(\E, \O_X) = 0 \ \hbox{for} \ i > 0.
\end{equation}
If $Z = \emptyset$, we have that $\SExt^i_{\O_X}(\I_{Z/X}(D), \O_X) = 0$ for $i > 0$, hence \eqref{vb} follows by applying $\SHom_{\O_X}(-,\O_X)$ to \eqref{est}. If $Z \ne \emptyset$, by (i) and diagram \eqref{diag20} we have that the map $W \otimes \O_X \to \omega_Z(-K_X-D)$ in \eqref{est20} is surjective and therefore $\SExt^1_{\O_X}(\E, \O_X)=0$. Moreover $\SExt^i_{\O_X}(\I_{Z/X}(D), \O_X) \cong \SExt^i_{\O_X}(\I_{Z/X}, \O_X)(-D)=0$ for $i \ge 2$ by \cite[Cor.~III.5.22]{alkl} since $Z$ is Cohen-Macaulay. Hence, applying $\SHom_{\O_X}(-,\O_X)$ to \eqref{est}, we get $\SExt^i_{\O_X}(\E, \O_X) = 0$ for $i \ge 2$. Thus \eqref{vb} is proved and $\E$ is locally free. 

To see that $\E$ is Ulrich, let $p \in \{1, \ldots, n\}$ and set $\delta_p=\delta_{Z, W, -pH}$. Now \eqref{est}, Kodaira vanishing, (iii) and (iv) imply that $H^i(\E(-p))=0$ for $0 \le i \le n-2$. Moreover, since $Z$ is either empty or of pure codimension $2$, we have that $H^n(\I_{Z/X}(D-pH)) \cong H^n(\O_X(D-pH)) \cong H^0(K_X+pH-D)=0$ by (ii). Therefore \eqref{est} and Kodaira vanishing give the exact sequence
$$\xymatrix{0 \ar[r] & H^{n-1}(\E(-p)) \ar[r] & H^{n-1}(\I_{Z/X}(D-pH)) \ar[r]^{\delta_p} & W^* \otimes H^n(\O_X(-pH)) \ar[r] & H^n(\E(-p)) \ar[r] & 0.}$$
We will now prove that \eqref{va1} holds. This will imply that $H^{n-1}(\E(-p))=H^n(\E(-p))=0$, hence that $\E$ is Ulrich.
To see \eqref{va1}, we first observe that 
\begin{equation}
\label{eq}
h^{n-1}(\I_{Z/X}(D-pH))=(r-1)h^n(\O_X(-pH)). 
\end{equation}
In fact, we have proved above that $H^n(\I_{Z/X}(D-pH))=0$, hence, using (iii)-(v), Kodaira vanishing and Serre duality
$$h^{n-1}(\I_{Z/X}(D-pH))=(-1)^{n-1}\chi(\I_{Z/X}(D-pH))=(r-1) \chi(K_X+pH)=(r-1)h^n(\O_X(-pH))$$
that is \eqref{eq}. It follows by \eqref{eq} and (vi) that $\delta_n$ is surjective, hence so is $\delta_p$ by the commutative diagram
\begin{equation}
\label{diag}
\xymatrix{H^{n-1}(\I_{Z/X}(D-nH)) \ar[d] \ar@{->>}[r]^{\delta_n} & W^* \otimes H^n(\O_X(-nH)) \ar@{->>}[d] \\ H^{n-1}(\I_{Z/X}(D-pH)) \ar[r]^{\delta_p} & W^* \otimes H^n(\O_X(-pH)).}
\end{equation}
Hence \eqref{eq} shows that $\delta_p$ is an isomorphism, so that \eqref{va1} holds. Thus we have proved that $\E$ is Ulrich. Also \eqref{est} implies that $\E$ has rank $r$ and $\det \E = \O_X(D)$.
\end{proof}
\renewcommand{\proofname}{Proof}

\begin{remark}
It follows by Definition \ref{ulrsub} that the subvariety $Z$ associated to $\E$ as in Theorem \ref{1-1} (and also in Corollaries \ref{Xacm+sottoc}, \ref{Xacm+sottoc+r=2} and \ref{hyper2}) is an Ulrich subvariety. In particular, by construction, $Z$ satisfies all properties in Lemma \ref{zeta}. 
\end{remark}

We now give a few remarks (some similar to \cite[Lemma 3.3]{cfk}), allowing to understand and use the properties of Ulrich subvarieties. To simplify the next statements, we will give the following

\begin{defi} 
Let $n \ge 2, d \ge 2, r \ge 2$ and let $D$ be a divisor on $X$. {\it A pair $(Z,W)$ satisfies $(\ast)$} if $\emptyset \ne Z \subset X$ is a subvariety satisfying conditions (a)-(c) of Theorem \ref{1-1}, and $W$ is a $(r-1)$-dimensional subspace $W \subseteq \Ext^1_{\O_X}(\I_{Z/X}(D), \O_X)$.
\end{defi} 
\begin{remark}
\label{primo}

Let $r \ge 2$, let $X \subset \P^N$ be a smooth irreducible variety of dimension $n \ge 2$ and degree $d \ge 2$ and let $D$ be a divisor on $X$. Given a pair $(Z,W)$ satisfying $(\ast)$, we consider the maps $\gamma_{Z,D}^*$ in \eqref{gammastar} and $\mu_{Z, W}$ in \eqref{mu}.

\begin{itemize}
\item[(i)] If $H^2(-D)=0, X$ is subcanonical, $Z \subset X$ is a nonempty Ulrich subvariety and $W$ is the associated $(r-1)$-dimensional subspace, then 
$$\gamma_{Z,D}^*(W)=H^0(\omega_Z(-K_X-D)).$$
\item[(ii)] Suppose that $H^1(-D)=H^2(-D)=H^{n-2}(D-nH)=H^{n-1}(D-nH)=0$ and $(Z,W)$ satisfies $(\ast)$. Then, the condition (vi) in Theorem \ref{1-1} is equivalent to $\mu_{Z, W}$ being either injective or surjective.
\item[(iii)] If $r=2, H^0(\I_{Z/X}(K_X+nH))=H^1(-D)=0, (Z,W)$ satisfies $(\ast)$, $Z$ is irreducible if $n \ge 3$, and, if $n=2$, the multiplication map $\mu: \gamma_{Z,D}^*(W) \otimes H^0(\O_Z(K_X+2H)) \to H^0(\omega_Z(2H-D))$
is either injective or surjective, then the condition (vi) in Theorem \ref{1-1} holds.
\item[(iv)] Assume that $n \ge 3, D=uH$ for some $u \in \Z$ and $X$ is subcanonical. Given $(Z,W)$ satisfying $(\ast)$, we have $\gamma_{Z,uH}^*(W) \subseteq H^0(\omega_Z(i_X-u))$ and the multiplication map 
$$\mu: \gamma_{Z,uH}^*(W) \otimes H^0(\O_Z(n-i_X)) \to H^0(\omega_Z(n-u)).$$ 
Then: 
\begin{itemize}
\item[(iv-a)] If $Z \subset X$ is a nonempty Ulrich subvariety, then $\gamma_{Z,uH}^*(W) = H^0(\omega_Z(i_X-u))$ and $\mu$ is either injective or surjective.
\item[(iv-b)] If $u>0, n-u-i_X \le -1$ and $H^i(\I_{Z/X}(n-i_X))=0$ for $i=0, 1$. Then the condition (vi) in Theorem \ref{1-1} holds if and only if $\mu$ is either injective or surjective.
\end{itemize}
\end{itemize}
\end{remark} 
\begin{proof}
To see (i), observe that Theorem \ref{1-1} gives an Ulrich bundle $\E$ sitting in the exact sequence \eqref{est2}. Then $H^1(\E^*) \cong H^{n-1}(\E(-i_X))=0$ by Lemma \ref{ulr}(ii). Now consider the splitting of \eqref{est2} as
$$0 \to \O_X(-D) \to \E^* \to \mathcal Q \to 0$$
and
$$0 \to \mathcal Q \to W \otimes \O_X \to \omega_Z(-K_X-D) \to 0.$$
We get that $H^1(\mathcal Q)=0$, hence, using \eqref{im}, we see that $\gamma_{Z,D}^*(W) = H^0(\omega_Z(-K_X-D))$ and (i) is proved. As for (ii), observe that since $H^{n-2}(K_X+D)=H^{n-1}(K_X+D)=0$, we have that $\gamma_{Z,D}$ is an isomorphism, hence so is $\gamma_{Z,D}^*$. Moreover, $H^{n-2}(D-nH)=H^{n-1}(D-nH)=0$ gives that the map $\alpha$ (see \eqref{alfa}) is an isomorphism, hence so is $\alpha^*$. Therefore (ii) follows by diagram \eqref{diam}.  To see (iii), since $H^1(-D)=0$ we have that $\gamma_{Z,D}$ is surjective, hence $\gamma_{Z,D}^*$ is injective and $\dim (\gamma_{Z,D}^*(W)) = \dim W = 1$. If $n \ge 3$, since $Z$ is irreducible, it follows that the multiplication map 
$$\mu: \gamma_{Z,D}^*(W) \otimes H^0(\O_Z(K_X+nH)) \to H^0(\omega_Z(nH-D))$$ 
is injective. The same holds if $n=2$ and $\mu$ is surjective since $H^0(\omega_Z(nH-D))$ and $\gamma_{Z,D}^*(W) \otimes H^0(\O_Z(K_X+nH)))$ have the same dimension, namely the cardinality of $Z$. Thus, in all cases, we have that $\mu$ is injective. Also $H^0(\I_{Z/X}(K_X+nH))=0$, so that the restriction map 
$$H^0(K_X+nH) \to H^0(\O_Z(K_X+nH))$$ 
is injective. Since $\mu_{Z, W}$ factorizes as 
\begin{equation}
\label{fact}
\xymatrix{\gamma_{Z,D}^*(W) \otimes H^0(K_X+nH) \ar[r] &  \gamma_{Z,D}^*(W) \otimes H^0(\O_Z(K_X+nH)) \ar[r]^{\hskip .9cm \mu} & H^0(\omega_Z(nH-D))}
\end{equation}
we get that $\mu_{Z, W}$ is injective. Now diagram \eqref{diam} shows that $\delta^*$ is injective, hence 
$\delta=\delta_{Z,W,-nH}$ is surjective. Thus (iii) holds. To see (iv-a), let $\E$ be the Ulrich bundle given by Theorem \ref{1-1}. We have by Lemma \ref{ulr}(iii) that $u=\frac{r(n+1-i_X)}{2}$. Since, as is well-known, $(n+1-i_X)H=K_X+(n+1)H$ is effective, we get that $n+1-i_X \ge 0$, hence $n-u-i_X \le -1$. Therefore $H^i(\E(K_X+nH-D))=H^i(\E(n-u-i_X))=0$ for $i=0, 1$ by Lemma \ref{ulr}(ii) and \eqref{est} implies that $H^i(\I_{Z/X}(K_X+nH))=0$ for $i=0, 1$. Hence the restriction map $H^0(K_X+nH) \to H^0(\O_Z(K_X+nH))$ is an isomorphism. Note that $H^j(-uH)=0$ for $j=1, 2$ by Kodaira vanishing because $u > 0$ by Lemma \ref{ulr}(v). Hence $\gamma_{Z,uH}^*(W)=H^0(\omega_Z(i_X-u))$ by (i). Moreover, for $j=n-2, n-1$, we have that $H^j(D-nH)=H^{n-j}(\O_X(n-u-i_X))=0$ by Kodaira vanishing, hence we conclude the proof of (iv-a) using (ii) and the factorization \eqref{fact}. As for (iv-b), by Kodaira vanishing, $H^i(-D)=H^i(-uH)=0$ for $i=1, 2$, $H^j(D-nH)=0$ for $j=n-2, n-1$ and we conclude applying (ii) and the factorization \eqref{fact} since, as above, the restriction map $H^0(K_X+nH) \to H^0(\O_Z(K_X+nH))$ is an isomorphism. This proves (iv-b), hence (iv).
\end{proof}

Next, we give some conditions that allow to show that $Z \ne \emptyset$ and irreducible.

\begin{remark}
\label{Znonv}

Let $r \ge 2$, let $X \subset \P^N$ be a smooth irreducible variety of dimension $n \ge 2$ and degree $d \ge 2$ and let $D$ be a divisor on $X$. Let $Z \subset X$ be a subvariety. 

Then $Z \ne \emptyset$ if one of the following holds:
\begin{itemize}
\item[(i)] There is an $(r-1)$-dimensional subspace $W \subseteq \Ext^1_{\O_X}(\I_{Z/X}(D), \O_X)$ and $h^1(-D) \le r-2$.
\item[(ii)] There is an $(r-1)$-dimensional subspace $W \subseteq \Ext^1_{\O_X}(\I_{Z/X}(D), \O_X)$ and $D$ is semiample with numerical dimension $\nu(D) \ge 2$.
\item[(iii)] $Z$ is an Ulrich subvariety and $\rho(X)=1$.
\item[(iv)] $Z$ is an Ulrich subvariety and the associated Ulrich bundle $\E$ as in Theorem \ref{1-1} is such that $(X, H, \E) \ne (\P(\F), \O_{\P(\F)}(1), \pi^*(\G(\det \F))$, where $\F$ is a very ample rank $n$ vector bundle on a smooth curve $C$ and $\G$ is a rank $2$ vector bundle on $C$ such that $H^i(\G)=0$ for $i \ge 0$, where $\pi : \P(\F) \to C$.
\end{itemize}
Moreover we have
\begin{itemize}
\item[(v)] If $Z$ is an Ulrich Gorenstein subvariety, $X$ does not contain lines, $\Pic(X) \cong \Z H$ and $n \ge 6, 3 \le r \le n-1$, then $Z$ is nonempty and smooth. 
\end{itemize}

Let $Z \subset X$ be a nonempty Ulrich subvariety and let $\E$ be the associated Ulrich bundle as in Theorem \ref{1-1}. Then $Z$ is irreducible if one of the following holds:
\begin{itemize}
\item[(vi)] $H^2(-D)=H^1(\E(-D))=0$.
\item[(vii)] $c_1(\E)^3 \ne 0$ \footnote{The cases $c_1(\E)^3=0$ can be classified as in \cite{ls}.} and either $c_1(\E)=uH$ for some $u \in \Z$ or $r=2$ and $X$ is subcanonical. 
\item[(viii)] $n \ge 3$ and $\E$ is $(n-3)$-ample (that is, by \cite[Thm.~1]{lr2}, $\E_{|M}$ does not have a trivial direct summand for every linear space $M \subseteq X \subset \P^N$ of dimension $n-2$).
\end{itemize}
Moreover we have:
\begin{itemize}
\item[(ix)] If $Z$ is smooth and irreducible, then $[K_Z(-K_X-D)]^{r-1}=0$.
\end{itemize}
\end{remark} 
\begin{proof}
To see (i), assume that $Z=\emptyset$. Then we have the contradiction
$$r-1=\dim W \le \dim \Ext^1_{\O_X}(\I_{Z/X}(D), \O_X)=h^1(-D).$$
Now (ii) follows by (i), since $h^1(-D)=0$ by \cite[Thm.~2]{mu}. As for (iv), assume that $Z=\emptyset$. We know that $\det \E$ is globally generated and non trivial by Lemma \ref{ulr}(v). It follows by (ii) that $\nu(\det \E)=1$ and therefore the image of $\Phi_{\E} : X \to \mathbb{G}(r-1,\P H^0(\E))$ is a curve and any nonempty fiber is a linear $\P^{n-1} \subset X \subseteq \P^N$ by \cite[Thm.~2]{ls}. It follows by \cite[Prop.~3.2.1]{bs} and \cite[Lemma 4.1]{lo} that $(X, H, \E)=(\P(\F), \O_{\P(\F)}(1), \pi^*(\G(\det \F))$, where $\F$ is a very ample rank $n$ vector bundle on a smooth curve $C$ and $\G$ is a rank $2$ vector bundle on $C$ such that $H^i(\G)=0$ for $i \ge 0$. This proves (iv). To see (iii) and (v), let $\E$ be the associated Ulrich bundle as in Theorem \ref{1-1}, so that $c_1(\E)=D>0$ by Lemma \ref{ulr}(v). Let $A$ be the ample generator of $\Num(X)$. Since $D \equiv uA$ for some $u \in \Z$, we have that $u>0$, hence $H^1(-D)=0$ by Kodaira vanishing and $Z \ne \emptyset$ by (i). This proves (iii). As for (v), we have that $Z \ne \emptyset$ by (iii) and $H^1(-D)=H^2(-D)=0$ by Kodaira vanishing. Therefore $\gamma_{Z,D}$ is an isomorphism, hence so is $\gamma_{Z,D}^*$. Thus, $\dim \gamma_{Z,D}^*(W)=r-1$ and $\gamma_{Z,D}^*(W)=H^0(\omega_Z(-K_X-D))$ by Remark \ref{primo}(i). It follows by property (i) in Theorem \ref{1-1} that $|\omega_Z(-K_X-D)|$ is base-point free of dimension $r-2$. Suppose now that $Z$ is singular. Since $\Sing(Z)=D_{r-3}(\varphi)$ (see proof of Lemma \ref{zeta}) and $\E$ is ample by \cite[Thm.~1]{ls}, it follows by \cite[Cor.~3.4(c)]{d} that $\Pic(Z) \cong \Z \O_Z(1)$. Also, $Z$ is Gorenstein and we deduce that $\omega_Z(-K_X-D) \cong \O_Z(a)$ for some $a \in \Z$. Since $H^1(\E(-D))=0$ as $\E$ is aCM by Lemma \ref{ulr}(ii), we get that $Z$ is irreducible by (vi). Now, $h^0(\O_Z(a))=r-1 \ge 2$, hence $a>0$. Therefore $\O_Z(a)$ is very ample and defines an embedding of $Z$ in $\P^{r-2}$. Hence $n-2= \dim Z \le r-2$, a contradiction. This proves (v).

Now assume that $Z \subset X$ is a nonempty Ulrich subvariety and let $\E$ be the associated Ulrich bundle as in Theorem \ref{1-1}. 

Under the hypotheses (vi) and (viii), we first show that $Z$ is connected. In fact, if $H^2(-D)=H^1(\E(-D))=0$, we deduce from \eqref{est} that $H^1(\I_{Z/X})=0$, hence that $Z$ is connected. If $n \ge 3$ and $\E$ is $(n-3)$-ample, then $Z$ is connected by \cite[Thm.~6.4(a)]{t}. Now, either $Z$ is smooth, therefore irreducible or $n \ge 6$ and $\dim \Sing(Z) = n-6$. In the latter case, $Z$ is again irreducible by Hartshorne's Connectedness Theorem \cite[Thm.~3.4]{h1}, \cite[Thm.~18.12]{e}. Hence (vi) and (viii) are proved. As for (vii), observe that if $c_1(\E)^3 \ne 0$, then $\nu(D)>2$, hence $H^2(-D)=0$ by \cite[Thm.~1.3]{wu} and \cite[Thm.~2]{ru} (see also \cite[Rmk.~11.2.20]{laz2}), because $D$ is globally generated, hence nef and abundant. Now, if $c_1(\E)=uH$, then $H^1(\E(-D))=H^1(\E(-u))=0$ by Lemma \ref{ulr}(ii). Also, if $r=2$ and $X$ is subcanonical, then $H^1(\E(-D))=H^1(\E^*)=H^{n-1}(\E(-i_X))=0$ by Lemma \ref{ulr}(ii). Thus (vii) follows by (vi). 
As for (ix), recall that $\gamma_{Z,D}^*(W)$ is a sublinear system of $|\omega_Z(-K_X-D)|$. Since $|\gamma_{Z,D}^*(W)|$ is base-point free by property (i) of Theorem \ref{1-1} and $\dim |\gamma_{Z,D}^*(W)| \le \dim |W| = r-2$ , it follows that $[K_Z(-K_X-D)]^{r-1}=0$. This proves (ix).
\end{proof}

\section{Application to arithmetically Gorenstein varieties}

In this section we will see how Theorem \ref{1-1} can be applied to arithmetically Gorenstein varieties. This will in turn give more specific properties of the corresponding Ulrich subvariety.

\begin{cor} 
\label{Xacm+sottoc}

Let $X \subset \P^N$ be a smooth irreducible aG variety of dimension $n \ge 2$ and degree $d \ge 2$. Then $(X,\O_X(1))$ carries a rank $r \ge 2$ Ulrich vector bundle $\E$ with $\det \E = \O_X(u)$ if and only if $u=\frac{r(n+1-i_X)}{2} \in \Z$ and there is an aCM subvariety $Z \subset \P^N$, contained in $X$, such that
\begin{itemize}
\item[(a)] $\dim Z = n-2$ and $Z$ is irreducible if $n \ge 3$,
\item[(b)] if $r=2$ or $n \le 5$, then $Z$ is smooth,
\item[(c)] if $n \ge 6$, then $Z$ is either smooth or is normal, Cohen-Macaulay, reduced and with $\dim \Sing(Z) = n-6$,
\end{itemize}
and, if $n=2$, there is a $(r-1)$-dimensional subspace $W \subseteq \Ext^1_{\O_X}(\I_{Z/X}(u), \O_X)$, such that the following hold:
\begin{itemize}
\item[(1)] $|\omega_Z(i_X-u)|$ is base-point free of dimension $r-2$ if $n \ge 3$ or $\gamma_{Z,uH}^*(W)$ generates $\omega_Z(i_X-u)$ if $n=2$.
\item[(2)] $H^0(\I_{Z/X}(u-1))=0$.
\item[(3)] $\chi(\O_Z(u-p))=\chi(\O_X(u-p))+(-1)^n(r-1) \chi(\O_X(p-i_X))$, for $1 \le p \le n$.
\item[(4)] If $n \ge 3$, the multiplication map $H^0(\omega_Z(i_X-u)) \otimes H^0(\O_Z(n-i_X)) \to H^0(\omega_Z(n-u))$ is either injective or surjective. If $n=2$, then $\delta_{Z, W, -2H} : H^1(\I_{Z/X}(u-2)) \to W^* \otimes H^2(-2H)$ is either injective or surjective.
\end{itemize}
Moreover the following exact sequence holds
\begin{equation}
\label{est3}
0 \to \O_X^{\oplus (r-1)} \to \E \to \I_{Z/X}(u) \to 0.
\end{equation}
\end{cor}
\begin{proof}
Assume that $\E$ is Ulrich with $\det \E = \O_X(u)$. Then $u=\frac{r(n+1-i_X)}{2} \in \Z$ by Lemma \ref{ulr}(iii) and $u> 0$ by Lemma \ref{ulr}(v). Let $(Z, W)$ be as in Theorem \ref{1-1}. In particular \eqref{est3} holds. Also, (2) is Theorem \ref{1-1}(iii) and (3) is equivalent to Theorem \ref{1-1}(v). As $H^1(\O_X(-u))=0$, it follows by Remark \ref{Znonv}(i) that $Z \ne \emptyset$. If $n \ge 3$, we have that $H^2(\O_X(-u))=0$ and $H^1(\E(-u))=0$ by Lemma \ref{ulr}(ii), hence Remark \ref{Znonv}(vi) gives that $Z$ irreducible. Also, when $n \ge 3$, since $H^{n-2}(\O_X(u-i_X)=H^{n-1}(\O_X(u-i_X))=0$, we have that $\gamma_{Z,uH}$ is an isomorphism, hence so is $\gamma_{Z,uH}^*$. Now Remark \ref{primo}(i) implies that $W \cong \gamma_{Z,uH}^*(W)=H^0(\omega_Z(i_X-u))$, hence the latter is base-point free by Theorem \ref{1-1}(i) and of dimension $r-1$. Thus we get (1). Also (4) follows by Theorem \ref{1-1}(vi) and Remark \ref{primo}(iv-a). It remains to prove that $Z$ is aCM.  As is well known, this holds if $n=2$. If $n \ge 3$, for any $j \in \Z$, we get by \eqref{est3} the exact sequences
$$0 \to \O_X(j-u)^{\oplus (r-1)} \to \E(j-u) \to \I_{Z/X}(j) \to 0.$$
Fix $i \in \{1,\ldots,n-2\}$ and $j \in \Z$. We have that $H^i(\E(j-u))=0$ by Lemma \ref{ulr}(ii) and also that $H^{i+1}(\O_X(j-u))=0$ since $X$ is aCM, hence $H^i(\I_{Z/X}(j))=0$. Now the exact sequence
\begin{equation}
\label{ide4}
0 \to \I_{X/\P^N}(j) \to \I_{Z/\P^N}(j) \to \I_{Z/X}(j) \to 0
\end{equation}
shows that $Z \subset \P^N$ is aCM. 

Vice versa, assume that we have a smooth irreducible aCM subvariety $Z \subset \P^N$, with $Z \subset X$, such that (a)-(c) hold and, if $n=2$, a $(r-1)$-dimensional subspace $W \subseteq \Ext^1_{\O_X}(\I_{Z/X}(u), \O_X)$, such that (1)-(4) hold with $u=\frac{r(n+1-i_X)}{2} \in \Z$. For $n \ge 3$, let $W=\Ext^1_{\O_X}(\I_{Z/X}(u), \O_X)$. Since $H^{n-2}(\O_X(u-i_X)=H^{n-1}(\O_X(u-i_X))=0$, we have that $\gamma_{Z,uH}$ is an isomorphism, hence so is $\gamma_{Z,uH}^*$, so that $\gamma_{Z,uH}^*(W)=H^0(\omega_Z(i_X-u))$. Then Theorem \ref{1-1}(i) holds by (1). Since, as is well-known, $(n+1-i_X)H=K_X+(n+1)H$ is effective and non trivial, we get that $n+1-i_X>0$, hence $u>0, n-u-i_X \le -1$. Therefore $H^0(K_X+nH-D)=H^0(\O_X(n-u-i_X))=0$, thus Theorem \ref{1-1}(ii) holds. Next, (2) is Theorem \ref{1-1}(iii) and (3) is equivalent to are Theorem \ref{1-1}(v). Since both $Z$ and $X$ are aCM, we get from \eqref{ide4} that $H^i(\I_{Z/X}(j))=0$ for $1 \le i \le n-2$ and for every $j \in \Z$, hence Theorem \ref{1-1}(iv) holds. Finally, since $u-1 \ge n-i_X$, we have that $H^0(\I_{Z/X}(n-i_X))=0$ by (2). Therefore, when $n \ge 3$, we get that the condition (vi) in Theorem \ref{1-1} is equivalent to (4) by Remark \ref{primo}(iv-b). It follows by Theorem \ref{1-1} that $Z$ and $W$ give rise to a rank $r \ge 2$ Ulrich vector bundle $\E$ with $\det \E = \O_X(u)$ and satisfying \eqref{est3}.
\end{proof}

In the case of rank $2$, Corollary \ref{Xacm+sottoc} and the properties of an Ulrich subvariety drastically simplify as follows.

\begin{cor} 
\label{Xacm+sottoc+r=2}

Let $X \subset \P^N$ be a smooth irreducible aG variety of dimension $n \ge 2$ and degree $d \ge 2$. Then $(X,\O_X(1))$ carries a rank $2$ Ulrich vector bundle $\E$ with $\det \E = \O_X(u)$ if and only if $u=n+1-i_X$ and there is a smooth aG subvariety $Z \subset \P^N$, irreducible when $n \ge 3$, with $Z \subset X, \dim Z = n-2, \omega_Z \cong \O_Z(n+1-2i_X)$ such that the following hold:
\begin{itemize}
\item[(a)] $H^0(\I_{Z/X}(n-i_X))=0$.
\item[(b)] $\chi(\O_Z(n+1-i_X-p))=\chi(\O_X(n+1-i_X-p))+(-1)^n\chi(\O_X(p-i_X))$, for $1 \le p \le n$.
\end{itemize}
Moreover the following exact sequence holds
\begin{equation}
\label{est3-bis}
0 \to \O_X \to \E \to \I_{Z/X}(n+1-i_X) \to 0.
\end{equation}
\end{cor}
\begin{proof}
Given $\E$ Ulrich of rank $2$ with $\det \E = \O_X(u)$, it follows by Corollary \ref{Xacm+sottoc} that $u=n+1-i_X$ and there is a smooth aCM subvariety $Z \subset \P^N$, irreducible when $n \ge 3$, with $Z \subset X, \dim Z = n-2$. By Lemma \ref{zeta}(iv) we have that $\omega_Z \cong \O_Z(n+1-2i_X)$, hence $Z$ is aG (see for example \cite[Prop.~3]{dpz}, \cite[Prop.~4.1.1]{mi}; when $n=2$ this can be proved also by  \cite[Prop.~4.1.10]{mi}, showing that $Z$ has the Cayley-Bacharach property and symmetric $h$-vector). Next, (a) and (b) are (2) and (3) of Corollary \ref{Xacm+sottoc}. 

Vice versa, assume that $u=n+1-i_X$ and that a subvariety $Z$ as in the statement and satisfying (a) and (b) is given. 

We will check that (1)-(4) of Corollary \ref{Xacm+sottoc} hold. Note that (a) and (b) are (2) and (3) of Corollary \ref{Xacm+sottoc}. 
As for (1) and (4), we consider two cases.

First, suppose that $n \ge 3$. Since $\omega_Z(2i_X-n-1) \cong \O_Z$ is base-point free, $(1)$ and $(4)$ hold.

Next, assume that $n=2$. Note that $i_X < 3$, for otherwise, as is well known $(X,H)=(\P^2,\O_{\P^2}(1))$ contradicting the fact that $d \ge 2$. By \cite[Prop.~2.1 and 2.2]{hss} we have that $3-2i_X={\rm reg}(\I_{Z/\P^N})-2$. Therefore $Z$ has the Cayley-Bacharach property with respect to $|\O_{\P^N}(3-2i_X)|$ by \cite[Thm.~4.1.10]{mi}. Since $X$ is projectively normal, we get that $Z$ has the Cayley-Bacharach property with respect to $|\O_X(3-2i_X)|$. Now \cite[Prop.~(1.33)]{gh} gives the existence of a rank $2$ vector bundle $\E$ on $X$ and a section $s \in H^0(\E)$ such that $Z = Z(s)$ and $\det \E = \O_X(3-i_X)$. Thus, setting $W= \langle s \rangle^*$ we have that the exact sequence \eqref{est} holds and we get an exact sequence
$$0 \to \O_X(i_X-3) \to \E^* \to \I_{Z/X} \to 0$$
that implies that $H^0(\E^*)=0$. Hence, as in the proof of Theorem \ref{1-1}, we deduce that $W \subseteq \Ext^1_{\O_X}(\I_{Z/X}(3-i_X), \O_X)$. Therefore, as in the beginning of Section \ref{corr}, dualizing \eqref{est} we get \eqref{est2}. Hence diagram \eqref{diag20} shows that $\gamma_{Z,(3-i_X)H}^*(W)$ generates $\omega_Z(2i_X-3)$ and therefore (1) of Corollary \ref{Xacm+sottoc} holds. Moreover this gives that the multiplication map 
$$\mu: \gamma_{Z,(3-i_X)H}^*(W) \otimes H^0(\O_Z(2-i_X)) \to H^0(\omega_Z(i_X-1)) = H^0(\O_Z(2-i_X))$$ 
is surjective. Therefore (4) holds by Remark \ref{primo}(iii). 
\end{proof}

Next, we specialize to the case of rank $2$ in hypersurfaces. The following is a slight improvement of \cite[Prop.~8.2]{b1} (see also \cite[Thm.~2.1]{cf} for the case $n=2$).

\begin{cor} 
\label{hyper2}

Let $n \ge 2$, let $S=\C[X_0,\ldots,X_{n+1}]$ and let $X \subset \P^{n+1}$ be a smooth hypersurface of degree $d \ge 2$. Then $(X,\O_X(1))$ carries a rank $2$ Ulrich vector bundle $\E$, with $\det \E = \O_X(d-1)$ if $n=2$, if and only if there is a smooth $(n-2)$-dimensional aG subvariety $Z \subset X$, irreducible if $n \ge 3$, such that $I_Z$ has the following minimal free resolution
\begin{equation}
\label{risolhom0}
0 \to S(-2d+1) \to S(-d)^{\oplus (2d-1)} \to S(-d+1)^{\oplus (2d-1)} \to I_Z \to 0.
\end{equation}
\end{cor}
\begin{proof}
Note that $i_X=n+2-d$, hence $n+1-i_X=d-1$. By Lefschetz's hyperplane theorem we have that $\Pic(X) \cong \Z H$ if $n \ge 3$. Let $\E$ be rank $2$ Ulrich vector bundle, so that $\det \E = \O_X(d-1)$ if $n \ge 3$ by \cite[Lemma 3.2]{lo}. Thus Corollary \ref{Xacm+sottoc+r=2} applies and we have a smooth, irreducible if $n \ge 3$, $(n-2)$-dimensional aG subvariety $Z \subset X$ with $\omega_Z \cong \O_Z(2d-n-3)$. It follows by \cite[\S 3, page 466]{be} that $I_Z$ has the following minimal free resolution 
$$0 \to S(-f) \to \bigoplus\limits_{i=1}^s S(-b_i)\to \bigoplus\limits_{j=1}^s S(-a_j) \to I_Z \to 0$$
where $a_1 \le \ldots \le a_s, b_i = f- a_i, 1 \le i \le s$ and $\omega_Z \cong \O_Z(f-n-2)$, so that $f=2d-1$ by \cite[Prop.~2.2]{hss}. By Corollary \ref{Xacm+sottoc+r=2}(a) we have that $H^0(\I_{Z/X}(d-2))=0$ and the exact sequence
\begin{equation}
\label{ide5}
0 \to \O_{\P^{n+1}}(-2) \to \I_{Z/\P^{n+1}}(d-2) \to \I_{Z/X}(d-2) \to 0
\end{equation} 
shows that $(I_Z)_{d-2}=H^0(\I_{Z/\P^{n+1}}(d-2))=0$. Therefore $a_1 \ge d-1$. 

We now claim that $a_i = d-1$ for all $1 \le i \le s$.

The exact sequence \eqref{est3-bis} and \cite[(3.1)]{b2} show that $h^0(\I_{Z/X}(d-1))= 2d-1$, hence the exact sequence
$$0 \to \O_{\P^{n+1}}(-1) \to \I_{Z/\P^{n+1}}(d-1) \to \I_{Z/X}(d-1) \to 0$$
shows that $\dim (I_Z)_{d-1}=h^0(\I_{Z/\P^{n+1}}(d-1))= 2d-1$. Therefore we have that $a_1= \ldots a_{2d-1} = d-1$, hence, in particular, $a_2=d-1$. It follows by \cite[(c), page 63]{htv} that $b_i > a_{s+2-i}$ for $1 \le i \le s$, hence $b_i > a_2=d-1$ for $1 \le i \le s$. Thus
$$d-1 = a_1 \le a_i = 2d-1-b_i \le d-1$$
that is $a_i = d-1$ for all $1 \le i \le s$. 

Hence we have proved that $a_i = d-1$ for all $1 \le i \le s$ and it follows that $b_i = d$ for all $1 \le i \le s$. This proves that \eqref{risolhom0} holds. 

Vice versa, assume that we have a smooth irreducible $(n-2)$-dimensional aG subvariety $Z \subset X$ such that \eqref{risolhom0} holds. Hence $H^0(\I_{Z/\P^{n+1}}(d-2))=0$, and then $H^0(\I_{Z/X}(d-2))=0$ by \eqref{ide5}. Also, \eqref{risolhom0} gives that $\chi(\O_Z(d-1-p))$ satisfies Corollary \ref{Xacm+sottoc+r=2}(b). Thus the latter corollary applies and we get a rank $2$ Ulrich vector bundle $\E$.
\end{proof}
We deduce Corollary \ref{hyper3} from this. A nice feature of this is that for $n \ge 5$, the proof that an Ulrich subvariety cannot exist is just one line! 

\renewcommand{\proofname}{Proof of Corollary \ref{hyper3}}
\begin{proof}
Let $n \ge 2$ and let $X \subset \P^{n+1}$ be a smooth hypersurface of degree $d \ge 2$. For $n=2$ assume also that $\Pic(X) \cong \Z H$. If there exists a rank $2$ Ulrich bundle $\E$ on $X$, then $X$ contains a subvariety $Z$ as in Corollary \ref{hyper2} since, for $n=2$, we have that $\det \E=\O_X(d-1)$ by \cite[Lemma 3.2]{lo}. Moreover, we note that $X$ contains a family (isomorphic to $\P H^0(\E)$) of dimension $2d-1$ of such $Z$'s: In fact, since $H^0(\E(-1))=0$, as in \cite[Prop.~1.3]{h3} we have that the map $\P H^0(\E) \to \Hilb(X)$ sending $[s] \in \P H^0(\E)$ to $Z=Z(s)$ is injective. It follows by \eqref{risolhom0} that $\I_{Z/\P^{n+1}}(d-1)$ is globally generated, hence \cite[Thm.~1.4]{os} gives that $n \le 4$. This proves (iv). As for (i)-(iii), let $U_{d,n} \subset |\O_{\P^{n+1}}(d)|$ be the open set of smooth hypersurfaces. Assume that the statement of the corollary is false. Then we have:
\begin{itemize}
\item[(a)] ($3 \le n \le 4$) for every nonempty open subset $U \subseteq U_{d,n}$, there is an $X \in U$ carrying a rank $2$ Ulrich bundle;
\item[(b)] ($n=2$) for every set of countably many proper closed subvarieties $W_i \subset U_{d,n}$, there is an $X \in U_{d,n} \setminus \bigcup\limits_{i \in \mathbb N} W_i$ carrying a rank $2$ Ulrich bundle.
\end{itemize}
Let $\H_{d,n}$ be the Hilbert scheme of subschemes $Z \subset \P^{n+1}$ with Hilbert polynomial 
$$P(m)=\binom{m+n+1}{n+1}-(2d-1)\binom{m-d+n+1}{n}-\binom{m-2d+n+2}{n+1}.$$
If $3 \le n \le 4$, let $\H_{d,n}'$ be the union of irreducible components of $\H_{d,n}$ containing a point that represents a smooth irreducible $(n-2)$-dimensional subvariety $Z$ that is zero locus of a general section of a rank $2$ Ulrich vector bundle on a smooth hypersurface of degree $d$ in $\P^{n+1}$. In particular $I_Z$ has a resolution as \eqref{risolhom0} by Corollary \ref{hyper2}. If $n=2$, let $\H_{d,2}'$ be the family of locally closed $0$-dimensional subschemes of $\P^3$ consisting of graded Gorenstein $\C[X_0,\ldots,X_3]$-algebra quotients with Hilbert function 
\begin{equation}
\label{hf}
h(m)=h^0(\O_{\P^3}(m))-(2d-1)h^0(\O_{\P^2}(m-d+1))-h^0(\O_{\P^3}(m-2d+1)). 
\end{equation}
Note that $\H_{d,2}'$ is irreducible (see for example \cite{kj, kmr}) and it contains, by Corollary \ref{hyper2}, a point that represents a smooth $0$-dimensional subvariety $Z$ that is zero locus of a general section of a rank $2$ Ulrich vector bundle on a very general smooth hypersurface of degree $d$ in $\P^3$. 

Consider the incidence correspondence
$$\mathcal I = \{([Z'],X) : Z' \subseteq X\} \subset \H_{d,n}' \times U_{d,n}$$
together with the two projections $\pi_1 : \mathcal I \to \H_{d,n}', \pi_2 : \mathcal I \to U_{d,n}$. We first prove that $\pi_2$ is dominant. Assume that it is not and let $W = \overline{\pi_2(\mathcal I)}$. If $3 \le n \le 4$, let $U=U_{d,n} \setminus W$, if $n=2$ let $W_i$ be the components of the Noether-Lefschetz locus $NL(d) \subset U_{d,2}$ and consider the proper closed subvarieties $\{W, W_i, i \in \mathbb N\}$ of $U_{d,2}$. Now, by (a) and (b), there is $X \in U$ or $X \in U_{d,2} \setminus W \cup \bigcup_{i \in \mathbb N} W_i$ carrying a rank $2$ Ulrich bundle and with $\Pic(X) \cong \Z H$, hence as said at the beginning of the proof, $X$ contains a subvariety $Z$ with $[Z] \in \H_{d,n}'$. But then $X$ lies in the image of $\pi_2$, a contradiction.

Since $\pi_2$ is dominant, we can decompose $\mathcal I = Y_1 \cup \ldots \cup Y_s \cup Y_{s+1} \cup \ldots \cup Y_t$ into irreducible components such that $Y_j$ dominates $U_{d,n}$ if and only if $1 \le j \le s$. For each $j \in \{1, \ldots s\}$, let $U_j$ be the dense open subset of $U_{d,n}$ such that any fiber of ${\pi_2}_{|Y_j}$ over the points in $U_j$ has dimension $\dim Y_j - \dim U_{d,n}$.

Now, for $3 \le n \le 4$, consider the nonempty open subset
$$U=\bigcap\limits_{j=1}^s U_j \cap \big(U_{d,n} \setminus \bigcup\limits_{j=s+1}^t \overline{\pi_2(Y_j)}\big) \subseteq U_{d,n}.$$
By (a), there is $X \in U$ that carries a rank $2$ Ulrich bundle $\E$, hence as said at the beginning of the proof, we have that $\P H^0(\E) \subseteq \pi_2^{-1}(X)$. Since $\pi_2^{-1}(X) \cap Y_j = \emptyset$ for $s+1 \le j \le t$, we deduce that
$$\P H^0(\E) \subseteq \bigcup\limits_{j=1}^s ({\pi_2}_{|Y_j})^{-1}(X)$$
hence that there is a $j_0 \in \{1, \ldots s\}$ such that 
\begin{equation}
\label{a1}
\P H^0(\E) \subseteq ({\pi_2}_{|Y_{j_0}})^{-1}(X). 
\end{equation}
Since $X \in U_{j_0}$, we deduce that
\begin{equation}
\label{a}
\dim U_{d,n}+2d-1 \le \dim Y_{j_0}.
\end{equation}
Next, for $n=2$, let $V_j=U_{d,2} \setminus U_j$ for $1 \le j \le s$, let $V'_j=\overline{\pi_2(Y_j)}$ for $s+1 \le j \le t$  and let $W_i$ be the components of the Noether-Lefschetz locus $NL(d) \subset U_{d,2}$. Consider the countably many proper closed subvarieties $\{V_j, 1 \le j \le s; V'_j, s+1 \le j \le t; W_i, i \in \mathbb N\}$ of $U_{d,2}$. By (b), there is $X \in U_{d,2}$ not lying in any of these subvarieties, carrying a rank $2$ Ulrich bundle $\E$ and with $\Pic(X) \cong \Z H$, hence as said at the beginning of the proof, we have that $\P H^0(\E) \subseteq \pi_2^{-1}(X)$. Now, exactly as before, we deduce that
there is a $j_1 \in \{1, \ldots s\}$ such that 
\begin{equation}
\label{b1}
\P H^0(\E) \subseteq ({\pi_2}_{|Y_{j_1}})^{-1}(X). 
\end{equation}
Since $X \in U_{j_1}$, we deduce that
\begin{equation}
\label{b}
\dim U_{d,2}+2d-1 \le \dim Y_{j_1}.
\end{equation}
Let now $Y \in \{Y_{j_0}, Y_{j_1}\}$. We give an upper bound on its dimension. By \eqref{a1} and \eqref{b1}, we have that $Y$ contains a pair $([Z],X)$ where $Z$ is a smooth, irreducible if $3 \le n \le 4$, $(n-2)$-dimensional subvariety that is zero locus of a general section of a rank $2$ Ulrich vector bundle on $X$. Let $[Z'] \in \pi_1(Y)$ be general. %Since $Z$ is aCM, so is $Z'$ by semicontinuity and we have, again by semicontinuity, that
%$$\binom{d+n+1}{n+1}=h^0(\I_{Z/\P^{n+1}}(d))+h^0(\O_Z(d)) \ge h^0(\I_{Z'/\P^{n+1}}(d))+h^0(\O_{Z'}(d)) = \binom{d+n+1}{n+1}$$
%so that $h^0(\I_{Z'/\P^{n+1}}(d))=h^0(\I_{Z/\P^{n+1}}(d))$. Therefore, using \eqref{risolhom0} when $3 \le n \le 4$ and \eqref{hf} when $n=2$, we get
%\begin{equation}
%\label{idea}
%h^0(\I_{Z'/\P^{n+1}}(d))=h^0(\I_{Z/\P^{n+1}}(d))=(n+1)(2d-1).
%\end{equation}
Then, using \eqref{risolhom0} when $3 \le n \le 4$ and \eqref{hf} when $n=2$, we get by semicontinuity that
\begin{equation}
\label{idea}
h^0(\I_{Z'/\P^{n+1}}(d)) \le h^0(\I_{Z/\P^{n+1}}(d))=(n+1)(2d-1).
\end{equation}

Moreover, when $3 \le n \le 4$, \eqref{risolhom0} gives
\begin{equation}
\label{o}
h^0(\O_Z(d-1))=\binom{d+n}{n+1}-2d+1
\end{equation}
while, when $n=2$, \eqref{hf} gives 
\begin{equation}
\label{o1}
h(d-1)=\binom{d+2}{3}-2d+1.
\end{equation}
Since $\pi_1^{-1}([Z'])$ is an open subset in $\P H^0(\I_{Z'/\P^{n+1}}(d))$, \eqref{idea} gives that
\begin{equation}
\label{fibra}
\dim (\pi_1^{-1}([Z'])) \le (n+1)(2d-1)-1.
\end{equation}
If $3 \le n \le 4$, it follows by \cite[Thm.~2.6]{kmr}, \eqref{risolhom0} and \eqref{o} that
$$\begin{aligned}
h^0(N_{Z/\P^{n+1}}) & = (2d-1)[\binom{d+n}{n+1}-2d+1]+\binom{2d-1}{2}\binom{n+2}{n+1}-(2d-1)\binom{d+n}{n+1}=\\
& =(2d-1)[(n+2)(d-1)-2d+1].
\end{aligned}$$
and then semicontinuity gives
$$\begin{aligned}
\dim \pi_1(Y) & \le \dim T_{[Z']} \pi_1(Y) \le \dim T_{[Z']} \H_{d,n}' = h^0(N_{Z'/\P^{n+1}}) \le \\
& \le h^0(N_{Z/\P^{n+1}}) = (2d-1)[(n+2)(d-1)-2d+1].
\end{aligned}$$
If $n=2$, we get by \cite[Rmk.~page 79]{kmr} and \eqref{o1}, that $\dim \pi_1(Y) \le \dim \H_{d,2}' = (2d-1)(2d-3)$.
Therefore, in both cases, we find that
\begin{equation}
\label{imma}
\dim \pi_1(Y) \le (2d-1)[(n+2)(d-1)-2d+1].
\end{equation} 
Now, using \eqref{fibra} and \eqref{imma}, we get
$$\dim Y = \dim (\pi_1^{-1}([Z'])) + \dim (\pi_1(Y)) \le nd(2d-1)-1.$$
It follows by \eqref{a} and \eqref{b} that
$$\binom{d+n+1}{n+1}-1+2d-1=\dim U_{d,n}+2d-1 \le nd(2d-1)-1$$
a contradiction for the given values of $d$ and $n$ in (i)-(iii). This proves (i)-(iii) and ends the proof of the corollary.
\end{proof}
\renewcommand{\proofname}{Proof}
We end the section with the following sample result. It shows that, in the case of hypersurfaces, in many cases, Ulrich subvarieties must be singular and in fact non-Gorenstein.

\begin{cor}
\label{hyper22}

Let $n \ge 6, r \ge 3$ and let $X \subset \P^{n+1}$ be a smooth hypersurface of degree $d \ge 2n$. Let $Z \subset X$ be an Ulrich subvariety and suppose that one of the following holds:
\begin{itemize}
\item[(i)] $X$ is general, $n \ge 7$.
\item[(ii)] $X$ is very general and $n=6$.
\end{itemize}
Then $\Sing(Z) \ne \emptyset$ and $\dim \Sing(Z)=n-6$. Moreover, if $r \le n-1$, then $Z$ is not Gorenstein.
\end{cor}
\begin{proof} 
Note that $Z$ is as in Corollary \ref{Xacm+sottoc}. First, assume that $Z$ is smooth. We have that $\Pic(Z) \cong \Z \O_Z(1)$ by Barth-Larsen's type theorems (see for example \cite[Thm.~2.2]{h2}) if $n \ge 7$ and by \cite[Prop.~8]{a} if $n=6$. On the other hand, in both cases (i)-(ii), we have that $X$ does not contain lines, hence the Ulrich bundle associated to $Z$ as in Corollary \ref{Xacm+sottoc} is ample by \cite[Thm.~1]{ls}. It follows by \cite[Thm.~2.2(b)]{ei} (or \cite[Cor.~3.4(c)]{d}) that $\rho(Z) \ge 2$, a contradiction. Thus $\Sing(Z) \ne \emptyset$ and $\dim \Sing(Z)=n-6$ by Corollary \ref{Xacm+sottoc}. Finally, assume that $r \le n-1$. We know that $\Pic(X) \cong \Z H$ by Lefschetz's theorem and then $Z$ is not Gorenstein by Remark \ref{Znonv}(v).
\end{proof}

\section{Complete intersections}

The goal of this section is to study Ulrich vector bundles on complete intersections. We first compute the necessary quantities related to Ulrich bundles on them and to Ulrich subvarieties.

We henceforth establish the following notation. 

Let $s \ge 1, n \ge 2$ and let $X \subset \P^{n+s}$ be a smooth complete intersection of hypersurfaces of degrees $(d_1, \ldots, d_s)$ with $d_i \ge 1, 1 \le i \le s$ and degree $d= \prod\limits_{i=1}^s d_i \ge 2$. Let 
$$S = \sum\limits_{i=1}^s d_i \ \hbox{and} \ S' =  \begin{cases} 0 & \hbox{if } s=1 \\ \sum\limits_{1 \le i < j \le s}d_id_j & \hbox{if } s \ge 2 \end{cases}.$$

\begin{lemma} 
\label{genci}

Let $s \ge 1, n \ge 3, r \ge 2$ and let $X \subset \P^{n+s}$ be a smooth complete intersection of hypersurfaces of degrees $(d_1, \ldots, d_s)$ with $d_i \ge 1, 1 \le i \le s$ and degree $d= \prod\limits_{i=1}^s d_i \ge 2$. Let $H \in |\O_X(1)|$. Let $\E$ be a rank $r$ Ulrich vector bundle on $X$ and let $Z \subset X$ be the associated Ulrich subvariety, as in Theorem \ref{1-1}. Then $Z$ is irreducible, of dimension $n-2$, smooth when $r=2$ or when $n \le 5$ and:
\begin{itemize}
\item[(i)] $K_X=(S-s-n-1)H$.
\item[(ii)] $c_2(X)=\left[\binom{n+s+1}{2}+S(S-s-n-1)-S'\right]H^2$.
\item[(iii)] $c_1(\E)=\frac{r}{2}(S-s)H$.
\item[(iv)] $\deg(Z)=\frac{rd}{24}\left[(3r-2)S^2-6(r-1)sS+3(r-1)s^2-s-2S'\right]$.
\item[(v)] $$\begin{aligned}[t]
\chi(\O_Z(m))=& \binom{m+n+s}{n+s}+(-1)^{n+1}rd\binom{\frac{r}{2}(S-s)-m-1}{n}+(-1)^{n+s}(r-1)\binom{\frac{r}{2}(S-s)-m-1}{n+s}+\\
& \hskip -2.35cm +\sum_{k=1}^s(-1)^{k+n+s}\sum_{1\le i_1<\ldots<i_k\le s}\left[\binom{d_{i_1}+\ldots+d_{i_k}-m-1}{n+s}+(r-1)\binom{d_{i_1}+\ldots+d_{i_k}+\frac{r}{2}(S-s)-m-1}{n+s}\right]
\end{aligned}$$
\end{itemize}
Moreover suppose that one of the following holds:
\begin{itemize}
\item[(1)] $n \ge 5$, or 
\item[(2)] $n=4$, $X$ is hyperplane section of a smooth complete intersection $X' \subset \P^{5+s}$ and $\E = \E'_{|X}$, where $\E'$ is a vector bundle on $X'$, or
\item[(3)] $n=4$, $X$ is very general and $(d_1, \ldots, d_s) \not\in \{(2, \underbrace{1, \ldots, 1}_\text{s-1}), s \ge 1; (2, 2, \underbrace{1, \ldots, 1}_\text{s-2}), s \ge 2\}$ (up to 

\vskip -.4cm \noindent permutation). 
\end{itemize}
Then
\begin{itemize}
\item[(vi)] $c_2(\E)=eH^2$ with $e=\frac{r}{24}\left[(3r-2)S^2-6(r-1)sS+3(r-1)s^2-s-2S'\right] \in \Z$.
\item[(vii)]  If $r=3$, then 

$\begin{aligned}[t]
c_2(Z)=&-\frac{1}{8}[49S^2-104sS-32(n+1)S+52s^2+32ns+35s+4n^2+12n+8+6S']H_Z^2+\\
&+(4S-4s-n-1)K_ZH_Z.
\end{aligned}$
\end{itemize}
\end{lemma}
\begin{proof} 
(i) and (ii) follow from the tangent and Euler sequence of $X \subset \P^{n+s}$. By Lefschetz's theorem (see for example \cite[Thm.~2.1]{h2}), we have that $\Pic(X) \cong \Z H$. Now the properties of $Z$ and (iii) follow by Corollary \ref{Xacm+sottoc}, while (iv) follows by (i)-(iii) and Lemma \ref{zeta}(i). Also, (v) follows by Lemma \ref{zeta}(iii) together with the Koszul resolution of $\I_{X/\P^{n+s}}$. As for (vi), we claim that under any of the hypotheses (1), (2) or (3), the following holds:
\begin{equation}
\label{meglio}
\exists e \in \Z \ \hbox{such that} \ c_2(\E) = eH^2. 
\end{equation}
In fact, if $n \ge 5$, we have by Lefschetz's theorem (see for example \cite[Thm.~2.1]{h2}) that $H^4(X,\Z) \cong \Z H^2$. Hence \eqref{meglio} holds under hypothesis (1), and under hypothesis (2) we have that $c_2(\E') = e(H')^2$ on $X'$, for some $e \in \Z$ and $H' \in |\O_{X'}(1)|$. Hence also $c_2(\E) = c_2(\E'_{|X})=eH^2$, so that \eqref{meglio} holds under hypothesis (2). Also, under hypothesis (3), we know again by Noether-Lefschetz's theorem (see for example \cite[Thm.~1.1]{sp}) that every algebraic cohomology class of codimension $2$ in $X$ is in $\Z H^2$. Since $[Z]=c_2(\E)$ by Lemma \ref{zeta}(b), we have that \eqref{meglio} holds under hypothesis (3).

Now, using (iv), we deduce by \eqref{meglio} that
$$e = \frac{r}{24}[(3r-2)S^2-6(r-1)sS+3(r-1)s^2-s-2S']$$
that is (vi). Finally (vii) follows by (i)-(iii), (vi) and Lemma \ref{zeta}(viii).
\end{proof}

We now prove our second main result.

\renewcommand{\proofname}{Proof of Theorem \ref{ci}}
\begin{proof} 
We can assume, from the start, that $X$ is not a quadric, since it is well-known that a quadric does not carry any rank $3$ Ulrich bundle (all Ulrich bundles are direct sums of spinor bundles, that have rank $2^{\lfloor \frac{n-1}{2} \rfloor}$) and that there are rank $2$ Ulrich bundles, namely the spinor bundles, only if $n=4$.

Set $s=c$ if $c \ge 4$, while $s=4$ if $1 \le c \le 3$ and, in the latter case, we see $X \subset \P^{n+4}$ as smooth complete intersection of hypersurfaces of degrees $(d_1, \ldots, d_4)$ with $d_j=1, c+1 \le j \le 4$. In this way we also have that $X \subset \P^{n+s}$ is a smooth complete intersection of hypersurfaces of degrees $(d_1, \ldots, d_s)$ with $s \ge 4, d_i \ge 1, 1 \le i \le s$ and $d= \prod_{i=1}^s d_i \ge 2$.

Assume that we have an Ulrich vector bundle of rank $r$ on $X$. If $n \ge 5$, taking hyperplane sections and using Lemma \ref{ulr}(iv), it will be enough to show that, on the $4$-dimensional section of $X_4$ of $X$, there are no Ulrich bundles of rank $r \le 3$. 

With an abuse of notation, let us call again $X$ the above $4$-dimensional section. Hence we have that $X \subset \P^{4+s}$ is a smooth complete intersection of hypersurfaces of degrees $(d_1, \ldots, d_s)$ with $s \ge 4, d_i \ge 1, 1 \le i \le s$ and $d= \prod_{i=1}^s d_i \ge 2$.

Let $\E$ be an Ulrich bundle of rank $r$ on $X$. Since $\Pic(X) \cong \Z H$ by Lefschetz's theorem, it follows by Lemma \ref{ulr}(vii) that $r \ge 2$. Let $Z \subset X$ be the associated smooth irreducible surface, as in Lemma \ref{genci}. 

Observe that, under hypothesis (a) of the theorem we have that condition (2) of Lemma \ref{genci} holds, while under hypothesis (b) of the theorem we have that condition (3) of Lemma \ref{genci} holds. In any case, we deduce that Lemma \ref{genci}(vi)-(vii) hold.

Suppose first that $r=2$. 

Setting $H_Z=H_{|Z}$, Corollary \ref{Xacm+sottoc+r=2} gives 
\begin{equation}
\label{canz}
K_Z=(2S-2s-5)H_Z
\end{equation}
and Lemma \ref{zeta}(vii), \eqref{canz} together with Lemma \ref{genci}(i), (ii) and (vi) give
\begin{equation}
\label{c2z}
c_2(Z)=\frac{1}{12}\left[120+115s+27s^2-120S-54sS+32S^2-10S'\right]H_Z^2.
\end{equation}
By Lemma \ref{genci}(v) we see that 
$$\chi(\O_Z)= 1-2d\binom{S-s-1}{4}+(-1)^{s}\binom{S-s-1}{s+4}+$$
$$+\sum_{k=1}^s(-1)^{k+s}\sum_{1\le i_1<\ldots<i_k\le s}\left[\binom{d_{i_1}+\ldots+d_{i_k}-1}{s+4}+\binom{d_{i_1}+\ldots+d_{i_k}+S-s-1}{s+4}\right].$$
In the notation \eqref{effe} of the appendix, this is just
\begin{equation}
\label{prima1}
\chi(\O_Z)=f_{s,2,0}(d_1,\ldots,d_s).
\end{equation}
On the other hand, Noether's formula $Z$, $\chi(\O_Z)=\frac{1}{12}[K_Z^2+c_2(Z)]$ gives, using \eqref{canz}, \eqref{c2z} and Lemma \ref{genci}(iv), that
$$\chi(\O_Z)=\frac{5d}{1728}[45s^4-180s^3S+288s^2S^2-216sS^3+64S^4+198s^3-612s^2S-36s^2S'+700sS^2+72sSS'$$
$$\hskip 1cm -288S^3-40S^2S'+181s^2-432sS-140sS'+336S^2+144SS'+4(S')^2-84s-168S'].$$
In the notation \eqref{gi} of the appendix, this is just
\begin{equation}
\label{seconda1}
\chi(\O_Z)=g_{4,s}(d_1,\ldots,d_s).
\end{equation}
Then, \eqref{prima1} and \eqref{seconda1} give that 
$$g_{4,s}(d_1,\ldots,d_s)-f_{s,2,0}(d_1,\ldots,d_s)=0.$$
In the notation of the appendix, this means, by Lemma \ref{gl4}(1), that
$$m_{1^s}(s)(d_1,\ldots,d_s)q_{s,8}(d_1,\ldots,d_s)=0$$
or, equivalently,
$$dq_{s,8}(d_1,\ldots,d_s)=0$$
contradicting Lemma \ref{cg}. 

This concludes the proof in the case $r=2$.

Next, assume that $r=3$.

By Riemann-Roch we see that
\begin{equation}
\label{hk}
K_ZH_Z=-2\chi(\O_Z(1))+2\chi(\O_Z)+\deg(Z).
\end{equation}
Now Lemma \ref{genci} gives, in the notation \eqref{effe} and \eqref{gi} of the functions in the appendix, that 
\begin{equation}
\label{chi}
\chi(\O_Z(m))=f_{s,3,m}(d_1,\ldots,d_s) \ \hbox{and} \ \deg(Z)=\delta_s(d_1,\ldots,d_s)
\end{equation}
and therefore, in the notation \eqref{gi}, \eqref{hk} becomes
\begin{equation}
\label{ack}
K_ZH_Z=-2f_{s,3,1}(d_1,\ldots,d_s)+2f_{s,3,0}(d_1,\ldots,d_s)+\delta_s(d_1,\ldots,d_s)=h_s(d_1,\ldots,d_s).
\end{equation}
On the other hand, we have by Remark \ref{Znonv}(ix) that $[K_Z-\frac{5}{2}(S-s-2)H_Z]^2=0$, so that, using \eqref{ack} and the notation \eqref{gi}, we get
\begin{equation}
\label{k2'}
\begin{aligned}
K_Z^2 &= 5(S-s-2)K_ZH_Z-\frac{25}{4}(S-s-2)^2\deg(Z)=\\
& =5(S-s-2)h_s(d_1,\ldots,d_s)-\frac{25}{4}(S-s-2)^2\delta_s(d_1,\ldots,d_s)= k_s(d_1,\ldots,d_s).
\end{aligned}
\end{equation}
Next, we get by Lemma \ref{genci}(vii), using also the notation \eqref{gi}, that
\begin{equation}
\label{c2z'}
\begin{aligned}
c_2(Z) & =(4S-4s-5)K_ZH_Z-\frac{1}{8}[49 S^2-104sS-160S+6S'+52s^2+163s+120]H_Z^2=\\
& = (4S-4s-5)h_s(d_1,\ldots,d_s)-\frac{1}{8}[49 S^2-104sS-160S+6S'+52s^2+163s+120]\delta_s(d_1,\ldots,d_s)= \\
& = c_s(d_1,\ldots,d_s).
\end{aligned}
\end{equation}
Hence \eqref{k2'}, \eqref{c2z'} and Noether's formula, using also the notation \eqref{gi},  give
\begin{equation}
\label{secondachi}
\chi(\O_Z)=\frac{1}{12}(K_Z^2+c_2(Z))=\frac{k_s(d_1,\ldots,d_s)+c_s(d_1,\ldots,d_s)}{12}=\chi'_s(d_1,\ldots,d_s).
\end{equation}
Thus we get, by \eqref{chi} and \eqref{secondachi} we have that 
$$\chi'_s(d_1,\ldots,d_s)-f_{s,3,0}(d_1,\ldots,d_s)=0$$
that is, using Lemma \ref{gl4}(2),
$$m_{1^s}(s)(d_1,\ldots,d_s)q_{s,9}(d_1,\ldots,d_s)=0$$
or, equivalently,
$$dq_{s,9}(d_1,\ldots,d_s)=0$$
contradicting Lemma \ref{cg}.

This concludes the proof in the case $r=3$ and therefore also ends the proof of the theorem.
\end{proof}
\renewcommand{\proofname}{Proof}

\eject

\appendix
\section{Symmetric functions associated to complete intersections}
\label{app}

Given a smooth complete intersection $X \subset \P^{s+4}$ of hypersurfaces of degrees $(d_1, \ldots, d_s)$, a rank $r \ge 2$ Ulrich vector bundle $\E$ on $X$ and an Ulrich subvariety $Z$, we have some natural symmetric functions of $(d_1, \ldots, d_s)$ as in Lemma \ref{genci} and in the proof of Theorem \ref{ci}. In this section we will lay out the necessary calculations related to them. Several calculations have been performed by Mathematica. The corresponding codes can be found in \cite{lr4}.

\begin{defi}
\label{effe}
Given integers $s \ge 1,r \ge 2, m$, consider the polynomials in $\Q[x_1,\ldots, x_s]$ given by 

$$a_s(m, x_1, \ldots, x_s) = \binom{m+s+4}{s+4}+\sum_{k=1}^s(-1)^{k+s}\sum_{1\le i_1<\ldots<i_k\le s}\binom{x_{i_1}+\ldots+x_{i_k}-m-1}{s+4}$$
and
$$b_s(x_1, \ldots, x_s)=-r(\prod_{i=1}^sx_i)\binom{\frac{r}{2}(\sum\limits_{i=1}^s x_i-s)-m-1}{4}.$$
Next we set 
$$f_{s,r,m}=a_s(m, x_1, \ldots, x_s)+(r-1)a_s(m-\frac{r}{2}(\sum\limits_{i=1}^s x_i-s), x_1, \ldots, x_s)+b_s(x_1, \ldots, x_s).$$
\end{defi}
Explicitly we have
$$\begin{aligned}[t]
& f_{s,r,m}= \binom{m+s+4}{s+4}-r(\prod_{i=1}^sx_i)\binom{\frac{r}{2}(\sum\limits_{i=1}^s x_i-s)-m-1}{4}+\\
& + (-1)^s(r-1)\binom{\frac{r}{2}(\sum\limits_{i=1}^s x_i-s)-m-1}{s+4}+\sum_{k=1}^s(-1)^{k+s}\sum_{1\le i_1<\ldots<i_k\le s}\binom{x_{i_1}+\ldots+x_{i_k}-m-1}{s+4}+\\
& +(r-1)\sum_{k=1}^s(-1)^{k+s}\sum_{1\le i_1<\ldots<i_k\le s}\binom{x_{i_1}+\ldots+x_{i_k}+\frac{r}{2}(\sum\limits_{i=1}^s x_i-s)-m-1}{s+4}.
\end{aligned}$$

We observe that the Koszul resolution of $\I_{X/\P^{s+4}}$ gives
\begin{equation}
\label{as}
\chi(\O_X(m))=a_s(m, d_1, \ldots, d_s)
\end{equation}
hence also
$$\chi(\O_X(m-\frac{r}{2}(\sum\limits_{i=1}^s d_i-s)))=a_s(m-\frac{r}{2}(\sum\limits_{i=1}^s d_i-s), d_1, \ldots, d_s).$$

The following properties of these function will be useful.

\begin{lemma}
\label{tf0}
\null \hskip 1cm
\begin{itemize}
\item[(1)] $f_{s,r,m}$ is symmetric in $x_1,\ldots, x_s$.
\item[(2)] For any $1 \le k \le s$, the following identity holds in $\mathbb{Q}[x_1,\ldots,x_k]$: 
\begin{equation}
\label{g1}
f_{k,r,m}(x_1, \ldots, x_k) = f_{s,r,m}(x_1,\ldots, x_k,1, \ldots, 1).
\end{equation}
\end{itemize}
\end{lemma}
\begin{proof}
It is clear that $b_s(x_1, \ldots, x_s)$ is symmetric and satisfies $b_k(x_1, \ldots, x_k) = b_s(x_1,\ldots, x_k,1, \ldots, 1)$ for any $1 \le k \le s$. Let $\pi \in S_s$ be a permutation. We have by \eqref{as} that $a_s(m, d_1, \ldots, d_s)=\chi(\O_X(m))=a_s(m, d_{\pi(1)}, \ldots, d_{\pi(s)})$, hence the polynomial $a_s(m, x_1, \ldots, x_s)-a_s(m, x_{\pi(1)}, \ldots, x_{\pi(s)})$ vanishes on all $(d_1, \ldots, d_s)$ and therefore it is zero. Thus $a(m, x_1, \ldots, x_s)$ is symmetric and so is $a(m-\frac{r}{2}(\sum\limits_{i=1}^s x_i-s), x_1, \ldots, x_s)$. It follows that $f_{s,r,m}$ is symmetric and (1) is proved. Also, again by \eqref{as},
$$a_s(d_1,\ldots, d_k,1, \ldots, 1)=\chi(\O_X(m))=a_k(d_1,\ldots, d_k)$$
and
$$a_s(m-\frac{r}{2}(\sum\limits_{i=1}^k d_i-k), d_1, \ldots, d_k,1, \ldots, 1))=\chi(\O_X(m-\frac{r}{2}(\sum\limits_{i=1}^k d_i-k)))=a_k(m-\frac{r}{2}(\sum\limits_{i=1}^k d_i-k), d_1, \ldots, d_k)$$
and therefore
$$f_{s,r,m}(d_1,\ldots, d_k,1, \ldots, 1)=f_{k,r,m}(d_1, \ldots, d_k).$$
Hence the polynomial $f_{s,r,m}(x_1,\ldots, x_k,1, \ldots, 1)-f_{k,r,m}(x_1, \ldots, x_k)$ vanishes on all positive integer $(d_1, \ldots, d_s)$, thus it is zero and this proves (2).
\end{proof}

\begin{lemma}
\label{tf1}
$x_i \mid f_{s,r,m}$ for all $1\le i\le s$.
\end{lemma}
\begin{proof}
By symmetry, it is enough to show that $x_1\mid f_{s,r,m}$. We view $f_{s,r,m}=\gamma(x_1)$ where $\gamma \in \Q[x_2,\ldots x_s][x_1]$, so that it is enough to prove that $\gamma(0)=0$.  Now
$$\begin{aligned}[t]
\gamma(0) & =\binom{m+s+4}{s+4}+(-1)^s(r-1)\binom{\frac{r}{2}(\sum\limits_{i=2}^s x_i-s)-m-1}{s+4} +\\
& +\sum_{i=1}^s(-1)^{s+1}\left[\binom{x_i-m-1}{s+4}+(r-1)\binom{x_i+\frac{r}{2}(\sum\limits_{i=2}^s x_i-s)-m-1}{s+4}\right]+\\
& +\sum\limits_{k=2}^s(-1)^{k+s}\sum_{1\le i_1<\ldots<i_k\le s}\binom{x_{i_1}+\ldots+x_{i_k}-m-1}{s+4}+\\
&+ (r-1)\sum\limits_{k=2}^s(-1)^{k+s}\sum_{1\le i_1<\ldots<i_k\le s}\binom{x_{i_1}+\ldots+x_{i_k}+\frac{r}{2}(\sum\limits_{i=2}^s x_i-s)-m-1}{s+4}.
\end{aligned}$$
Since we are setting $x_1=0$, in the sum in the second line the term with $i=1$ is
$$(-1)^{s+1}\left[\binom{-m-1}{s+4}+(r-1)\binom{\frac{r}{2}(\sum\limits_{i=2}^s x_i-s)-m-1}{s+4}\right]=$$
$$=(-1)^{s+1}\left[(-1)^s\binom{m+s+4}{s+4}+(r-1)\binom{\frac{r}{2}(\sum\limits_{i=2}^s x_i-s)-m-1}{s+4}\right]$$
thus it cancels with the term in the first line 
$$\binom{m+s+4}{s+4}+(-1)^s(r-1)\binom{\frac{r}{2}(\sum\limits_{i=2}^s x_i-s)-m-1}{s+4}.$$
Hence we get
$$\begin{aligned}[t]
\gamma(0) & =\sum\limits_{i=2}^s(-1)^{s+1}\left[ \binom{x_i-m-1}{s+4}+(r-1)\binom{x_i+\frac{r}{2}(\sum\limits_{i=2}^s x_i-s)-m-1}{s+4}\right]+\\
& +\sum\limits_{k=2}^s(-1)^{k+s}\sum_{1\le i_1<\ldots<i_k\le s}\binom{x_{i_1}+\ldots+x_{i_k}-m-1}{s+4}+\\
&+(r-1)\sum\limits_{k=2}^s(-1)^{k+s}\sum_{1\le i_1<\ldots<i_k\le s}\binom{x_{i_1}+\ldots+x_{i_k}+\frac{r}{2}(\sum\limits_{i=2}^s x_i-s)-m-1}{s+4}=\\
& =\sum\limits_{i=2}^s(-1)^{s+1}\left[ \binom{x_i-m-1}{s+4}+(r-1)\binom{x_i+\frac{r}{2}(\sum\limits_{i=2}^s x_i-s)-m-1}{s+4}\right]+\\
& +\sum\limits_{k=2}^s(-1)^{k+s}\sum_{2\le i_2<\ldots<i_k\le s}\binom{x_{i_2}+\ldots+x_{i_k}-m-1}{s+4}+\\
& +(r-1)\sum\limits_{k=2}^s(-1)^{k+s}\sum_{2\le i_2<\ldots<i_k\le s}\binom{x_{i_2}+\ldots+x_{i_k}+\frac{r}{2}(\sum\limits_{i=2}^s x_i-s)-m-1}{s+4}+\\
& +\sum\limits_{k=2}^s(-1)^{k+s}\sum_{2\le i_1<\ldots<i_k\le s}\binom{x_{i_1}+\ldots+x_{i_k}-m-1}{s+4}+\\
&+(r-1)\sum\limits_{k=2}^s(-1)^{k+s}\sum_{2\le i_1<\ldots<i_k\le s}\binom{x_{i_1}+\ldots+x_{i_k}+\frac{r}{2}(\sum\limits_{i=2}^s x_i-s)-m-1}{s+4}=\\
& =\sum\limits_{i=2}^s(-1)^{s+1}\left[ \binom{x_i-m-1}{s+4}+(r-1)\binom{x_i+\frac{r}{2}(\sum\limits_{i=2}^s x_i-s)-m-1}{s+4}\right]+\\
& +\sum\limits_{i_2=2}^s(-1)^{2+s}\left[\binom{x_{i_2}-m-1}{s+4}+(r-1)\binom{x_{i_2}+\frac{r}{2}(\sum\limits_{i=2}^s x_i-s)-m-1}{s+4}\right]+\\
& +\sum\limits_{k=3}^s(-1)^{k+s}\sum_{2\le i_2<\ldots<i_k\le s}\binom{x_{i_2}+\ldots+x_{i_k}-m-1}{s+4}+\\
& +(r-1)\sum\limits_{k=3}^s(-1)^{k+s}\sum_{2\le i_2<\ldots<i_k\le s}\binom{x_{i_2}+\ldots+x_{i_k}+\frac{r}{2}(\sum\limits_{i=2}^s x_i-s)-m-1}{s+4}+\\
& +\sum\limits_{k=2}^s(-1)^{k+s}\sum_{2\le i_1<\ldots<i_k\le s}\binom{x_{i_1}+\ldots+x_{i_k}-m-1}{s+4}+\\
&+(r-1)\sum\limits_{k=2}^s(-1)^{k+s}\sum_{2\le i_1<\ldots<i_k\le s}\binom{x_{i_1}+\ldots+x_{i_k}+\frac{r}{2}(\sum\limits_{i=2}^s x_i-s)-m-1}{s+4}.
\end{aligned}$$
While the first two sums above cancel with each other, notice that in the last two sums the case $k=s$ is not possible, hence, rescaling it, we get
$$\begin{aligned}[t]
\gamma(0) = & \sum\limits_{k=3}^s(-1)^{k+s}\sum_{2\le i_2<\ldots<i_k\le s}\binom{x_{i_2}+\ldots+x_{i_k}-m-1}{s+4}+\\
&+(r-1)\sum\limits_{k=3}^s(-1)^{k+s}\sum_{2\le i_2<\ldots<i_k\le s}\binom{x_{i_2}+\ldots+x_{i_k}+\frac{r}{2}(\sum\limits_{i=2}^s x_i-s)-m-1}{s+4}+\\
& +\sum\limits_{k=3}^s(-1)^{k-1+s}\sum_{2\le i_1<\ldots<i_{k-1}\le s}\binom{x_{i_1}+\ldots+x_{i_{k-1}}-m-1}{s+4}+\\
&+(r-1)\sum\limits_{k=3}^s(-1)^{k-1+s}\sum_{2\le i_1<\ldots<i_{k-1}\le s}\binom{x_{i_1}+\ldots+x_{i_{k-1}}+\frac{r}{2}(\sum\limits_{i=2}^s x_i-s)-m-1}{s+4}
\end{aligned}$$
and the latter is $0$.
\end{proof}

We will now express the symmetric polynomials $f_{s,r,m}$ in terms of monomial symmetric polynomials. For this we will use some properties of them, for which we refer for example to \cite[\S 1]{eg}.

\begin{defi}
Let $s \ge 1$ be an integer and let $x_1,\ldots, x_s$ be indeterminates. Given a partition $\lambda = \{\lambda_1, \ldots, \lambda_k\}$ with $\lambda_1 \ge \lambda_2 \ge \ldots \ge \lambda_k \ge 1$, if $k \le s$ we let $m_{\lambda}(s)$ be the monomial symmetric polynomial in $x_1,\ldots,x_s$ corresponding to $\lambda$, while if $k > s$ we set $m_{\lambda}(s)=0$.
\end{defi}
We will also write $m_{\lambda}(s)=m_{\lambda_1 \ldots \lambda_k}$. We denote by $\{1^k\}$ the partition $\{1,\ldots,1\}$ of $k$ and we set $m_{1^0}(s)=1$. For example
$$m_h(s)=\sum\limits_{i=1}^s x_i^h \ \hbox{for} \ h \ge1 \ \hbox{and} \ m_{1^s}(s)=\prod_{i=1}^sx_i.$$

We will consider below the following $\Q$-basis of the vector space of symmetric polynomials with rational coefficients and of degree at most $4$ in $s$ variables: 
\begin{equation}
\label{base}
\{m_4(s), m_{31}(s), m_{22}(s), m_{211}(s), m_{1111}(s), m_3(s), m_{21}(s), m_{111}(s), m_2(s), m_{11}(s), m_1(s), 1\}.
\end{equation}
We will need some elementary relations among the $m_{\lambda}(s)$'s.

\begin{lemma}
\label{tf2}
The following identities hold:
\begin{itemize}
\item[(1)] $m_1(s)^2=m_2(s)+2m_{11}(s)$.
\item[(2)] $m_1(s)^3=m_3(s)+3m_{21}(s)+6m_{111}(s)$.
\item[(3)] $m_1(s)^4=m_4(s)+4m_{31}(s)+6m_{22}(s)+12m_{211}(s)+24m_{1111}(s)$.
\item[(4)] $m_1(s)m_{11}(s)=m_{21}(s)+3m_{111}(s)$.
\item[(5)] $m_1(s)^2m_{11}(s)=m_{31}(s)+2m_{22}(s)+5m_{211}(s)+12m_{1111}(s)$.
\item[(6)] $m_{11}(s)^2=m_{22}(s)+2m_{211}(s)+6m_{1111}(s)$.
\item[(7)] $m_1(s)m_3(s))=m_4(s)+m_{31}(s)$.
\item[(8)] $m_1(s)m_{21}(s)=m_{31}(s)+2m_{22}(s)+2m_{211}(s)$.
\item[(9)] $m_1(s)m_{111}(s)=m_{211}(s)+4m_{1111}(s)$.
\item[(10)] $m_1(s)m_2(s)=m_3(s)+m_{21}(s)$.
\item[(11)] $m_1(s)^2m_2(s)=m_4(s)+2m_{31}(s)+2m_{22}(s)+2m_{211}(s)$.
\item[(12)] $m_2(s)^2=m_4(s)+2m_{22}(s)$.
\item[(13)] $m_2(s)m_{11}(s)=m_{31}(s)+m_{211}(s)$.
\end{itemize}
\end{lemma}
\begin{proof}
Follows by straightforward computations.
\end{proof}
The next lemma will be very useful.
\begin{lemma}
\label{tf2-bis}
Let $G=G(x_1,\ldots, x_s)$ be a symmetric polynomial in $s \ge 4$ variables with $\deg G \le 4$. Then there are $a_i \in \Q$ such that 
\begin{equation}
\label{lc}
\begin{split}
G=& a_1m_4(s)+a_2m_{31}(s)+a_3m_{22}(s)+a_4m_{211}(s)+a_5m_{1111}(s)+a_6m_3(s)+a_7m_{21}(s)+a_8m_{111}(s)+ \\
& +a_9m_2(s)+a_{10}m_{11}(s)+a_{11}m_1(s)+a_{12}
\end{split}
\end{equation}
and the following identity holds:
\begin{equation}
\label{rel}
\begin{split}
& G(x_1,x_2,x_3,x_4,1,\ldots,1)= a_1m_{4}(4)+a_2m_{31}(4)+a_3 m_{22}(4)+a_4 m_{211}(4)+a_5 m_{1111}(4)+\\
& +[a_6+(s-4)a_2] m_{3}(4)+[a_7+(s-4)a_4] m_{21}(4)+[a_8+(s-4)a_5] m_{111}(4)+\\
& +\left[a_9+(s-4)(a_3+a_7)+\binom{s-4}{2}a_4\right] m_{2}(4)+\left[a_{10}+(s-4)(a_4+a_8)+\binom{s-4}{2}a_5\right] m_{11}(4)+\\
& +\left[a_{11}+(s-4)(a_2+a_7+a_{10})+\binom{s-4}{2}(2a_4+a_8)+\binom{s-4}{3}a_5\right] m_{1}(4)+a_{12}+\\
& +(s-4)(a_1+a_6+a_9+a_{11})+\binom{s-4}{2}(2a_2+a_3+2a_7+a_{10})+\binom{s-4}{3}(3a_4+a_8)+\binom{s-4}{4}a_5.
\end{split}
\end{equation}
\end{lemma}
\begin{proof}
The fact that $G$ can be written as \eqref{lc} follows by using the basis \eqref{base}. Then \eqref{rel} follows by the straightforward identities:
$$m_i(s)(x_1,x_2,x_3,x_4,1,\ldots,1)=m_i(4)+s-4, 1 \le i \le 4$$ 
$$m_{1^i}(s)(x_1,x_2,x_3,x_4,1,\ldots,1)=\sum\limits_{j=0}^i \binom{s-4}{j}m_{1^{i-j}}(4), 1 \le i \le 4$$ 
$$m_{31}(s)(x_1,x_2,x_3,x_4,1,\ldots,1)=m_{31}(4)+(s-4)m_3(4)+(s-4)m_1(4)+(s-4)(s-5)$$
$$m_{22}(s)(x_1,x_2,x_3,x_4,1,\ldots,1)=m_{22}(4)+(s-4)m_2(4)+\binom{s-4}{2}$$
$$m_{211}(s)(x_1,x_2,x_3,x_4,1,\ldots,1)=m_{211}(4)+(s-4)m_{21}(4)+\binom{s-4}{2}m_2(4)+(s-4)m_{11}(4)+$$
$$\hskip 1.3cm +(s-4)(s-5)m_1(4)+(s-4)\binom{s-5}{2}$$
$$m_{21}(s)(x_1,x_2,x_3,x_4,1,\ldots,1)=m_{21}(4)+(s-4)m_2(4)+(s-4)m_1(4)+(s-4)(s-5).$$
\end{proof}

We can now express some $f_{s,r,m}$ in terms of monomial symmetric polynomials. 

\begin{lemma}
\label{gl1} 
For all $s \ge 4$ the following identities hold:
\begin{itemize}
\item[(1)] $\begin{aligned}[t]
f_{s,2,0}=& \frac{m_{1^s}(s)}{360}\big[66m_4(s)+225m_{31}(s)+320m_{22}(s)+600m_{211}(s)+1125m_{1111}(s)\\
& -75(3s+4)m_3(s)-150(4s+5)m_{21}(s)-225(5s+6)m_{111}(s)+\\
& +10(30s^2+73s+35)m_2(s)+\frac{75(15s^2+35s+14)}{2}m_{11}(s)\\
&-\frac{75s(s+1)(5s+12)}{2}m_1(s)+\frac{s}{8}(375s^3+1650s^2+1505s-698)\big].
\end{aligned}$        
\item[(2)] $\begin{aligned}[t]
f_{s,3,0}=& \frac{m_{1^s}(s)}{1920}[1683m_4(s)+6060m_{31}(s)+8770 m_{22}(s)+16860 m_{211}(s)+32400 m_{1111}(s)\\
& -60(101s+95) m_{3}(s)-60(281s+255) m_{21}(s)-3600(9s+8) m_{111}(s)\\
& +10(843s^2+1496s+490) m_{2}(s)+60(270s^2+469s+140)m_{11}(s)\\
& -60s(90s^2+229s+125) m_{1}(s)+s(1350s^3+4470s^2+3305s-698)].
\end{aligned}$
\item[(3)] $\begin{aligned}[t]    
f_{s,3,1}=& \frac{m_{1^s}(s)}{1920}[1683m_4(s)+6060m_{31}(s)+8770 m_{22}(s)+16860 m_{211}(s)+32400 m_{1111}(s)\\
& -60(101s+133) m_{3}(s)-60(281s+357) m_{21}(s)-720(45s+56) m_{111}(s)+\\
& +10(843s^2+2108s+994) m_{2}(s)+60(270s^2+661s+284)m_{11}(s)\\
& -60s(90s^2+325s+263) m_{1}(s)+s(1350s^3+6390s^2+7265s-1418)].
\end{aligned}$
\end{itemize}
\end{lemma}
\begin{proof}
Let $f \in \{f_{s,2,0}, f_{s,3,0}, f_{s,3,1}\}$. Since $\Q[x_1,\ldots,x_s]$ is a UFD, using Lemma  \ref{tf1}, we see that there exists a $p_s=p_{s,r,m} \in \Q[x_1,\ldots,x_s]$ such that
\begin{gather*}
f=\frac{m_{1^s}(s)}{360}p_s \ \hbox{if} \ f=f_{s,2,0}, \ \hbox{and} \ f=\frac{m_{1^s}(s)}{1920}p_s \ \hbox{in the other cases}.
\end{gather*}
Moreover all $p_s$'s are symmetric by Lemma \ref{tf0}(i) and have degree at most $4$ since $\deg f \le s+4$. Thus, we can express all $p_s$'s through the basis \eqref{base} as follows:
\begin{equation}
\label{bas}
\begin{split}
p_s=& a_1m_4(s)+a_2m_{31}(s)+a_3m_{22}(s)+a_4m_{211}(s)+a_5m_{1111}(s)+ a_6m_3(s)+\\
& +a_7m_{21}(s)+a_8m_{111}(s)+a_9m_2(s)+a_{10}m_{11}(s)+a_{11} m_1(s)+a_{12}
\end{split}
\end{equation}
with $a_1,\ldots,a_{12}\in \Q$. 

Now observe that, applying Lemma \ref{tf0}(2), we have in $\Q[x_1, \ldots, x_4]$ the similar identity 
\begin{equation}
\label{g1-bis}
p_4(x_1,\ldots, x_4)= p_s(x_1,\ldots, x_4, 1 \ldots, 1).
\end{equation}
On the other hand, direct calculations show that:
\begin{equation*}
\begin{split}
f_{4,2,0}(x_1,x_2,x_3,x_4)=&\frac{m_{1^4}(4)}{360} [66m_{4}(4)+225m_{31}(4)+320 m_{22}(4)+600 m_{211}(4)+1125 m_{1111}(4)\\
& -1200 m_{3}(4)-3150 m_{21}(4)-5850 m_{111}(4)+8070 m_{2}(4)+14775m_{11}(4)\\
& -24000 m_{1}(4)+27861]\\
f_{4,3,0}(x_1,x_2,x_3,x_4)=&\frac{m_{1^4}(4)}{1920} [1683m_{4}(4)+6060m_{31}(4)+8770 m_{22}(4)+16860 m_{211}(4)+32400 m_{1111}(4)\\
& -29940 m_{3}(4)-82740 m_{21}(4)-158400 m_{111}(4)+199620 m_{2}(4)+380160 m_{11}(4)\\
& -595440 m_{1}(4)+681768]\\
f_{4,3,1}(x_1,x_2,x_3,x_4)=&\frac{m_{1^4}(4)}{1920} [1683m_{4}(4)+6060m_{31}(4)+8770 m_{22}(4)+16860 m_{211}(4)+32400 m_{1111}(4)\\
& -32220 m_{3}(4)-88860 m_{21}(4)-169920 m_{111}(4)+229140 m_{2}(4)+434880 m_{11}(4)\\
& -720720 m_{1}(4)+865128].
\end{split}
\end{equation*}
so that
\begin{equation}
\label{g2}
\begin{split}
p_{4,2,0}=&66m_{4}(4)+225m_{31}(4)+320 m_{22}(4)+600 m_{211}(4)+1125 m_{1111}(4)-1200 m_{3}(4)\\
& -3150 m_{21}(4)-5850 m_{111}(4)+8070 m_{2}(4)+14775m_{11}(4)-24000 m_{1}(4)+27861\\
p_{4,3,0}=&1683m_{4}(4)+6060m_{31}(4)+8770 m_{22}(4)+16860 m_{211}(4)+32400 m_{1111}(4)\\
& -29940 m_{3}(4)-82740 m_{21}(4)-158400 m_{111}(4)+199620 m_{2}(4)+380160 m_{11}(4)\\
& -595440 m_{1}(4)+681768\\
p_{4,3,1}=&1683m_{4}(4)+6060m_{31}(4)+8770 m_{22}(4)+16860 m_{211}(4)+32400 m_{1111}(4)\\
& -32220 m_{3}(4)-88860 m_{21}(4)-169920 m_{111}(4)+229140 m_{2}(4)+434880 m_{11}(4)\\
& -720720 m_{1}(4)+865128.
\end{split}
\end{equation}
Now replacing $x_5=\ldots=x_s=1$ in \eqref{bas} and using \eqref{rel} for $G=p_s$ we get an expression for $p_s(x_1,\ldots, x_4, 1 \ldots, 1)$ in terms of the basis \eqref{base} whose coefficients must coincide, by \eqref{g1-bis}, with the ones in \eqref{g2}. Solving the corresponding linear system in the $a_j$'s, we get (1)-(3). 
\end{proof}
Consider now the following polynomials in $\Q[x_1,\ldots, x_s]$:
\begin{equation}
\label{gi}
\begin{split}
g_{4,s}=& \frac{5m_{1^s}(s)}{1728}[45s^4-180s^3m_{1}(s)+288s^2m_{1}(s)^2-216sm_{1}(s)^3+64m_{1}(s)^4+198s^3-612s^2m_{1}(s)\\
&-36s^2m_{11}(s)+700sm_{1}(s)^2 +72sm_{1}(s)m_{11}(s)-288m_{1}(s)^3-40m_{1}(s)^2m_{11}(s)+181s^2\\
& -432sm_{1}(s)-140sm_{11}(s)+336m_{1}(s)^2+144m_{1}(s)m_{11}(s)+4m_{11}(s)^2-84s-168m_{11}(s)]\\
\delta_s=&\frac{m_{1^s}(s)}{8}[7m_1(s)^2-12sm_1(s)-2m_{11}(s)+6s^2-s]\\
h_s=& -2f_{s, 3, 1}+2f_{s, 3, 0}+\delta_s\\
k_s=& 5(m_{1}(s)-s-2)h_s-\frac{25}{4}(m_{1}(s)-s-2)^2\delta_s\\
c_s=& [4m_{1}(s)-4s-5]h_s-\frac{1}{8}[49m_1(s)^2-8(13s+20)m_1(s)+6m_{11}(s)+52s^2+163s+120]\delta_s\\
\chi'_s=& \frac{k_s+c_s}{12}.
\end{split}
\end{equation}
Then we have
\begin{lemma}
\label{gl2} 
For all $s \ge 1$ the following identities hold:
\begin{enumerate}
\item $\begin{aligned}[t]
g_{4,s}=& \frac{5m_{1^s}(s)}{1728}[64m_4(s)+216m_{31}(s)+308m_{22}(s)+576m_{211}(s)+1080m_{1111}(s)\\
& -72(3s+4)m_{3}(s)-144(4s+5)m_{21}(s)-216(5s+6)m_{111}(s)+\\
& +4(72s^2+175s+84)m_{2}(s)+36(15s^2+35s+14)m_{11}(s)-36s(s+1)(5s+12)m_{1}(s)+\\
& + s(3s-1)(3s+7)(5s+12)].
\end{aligned}$
\item $\delta_s=\frac{m_{1^s}(s)}{8}[7m_{2}(s)+12m_{11}(s)-12sm_{1}(s)+6s^2-s]$
\item $\begin{aligned}[t]
h_s =& \frac{m_{1^s}(s)}{8}[19 m_{3}(s)+51 m_{21}(s)+96 m_{111}(s)-(51s+35) m_{2}(s)-12(8s+5)m_{11}(s)+\\
& +3s(16s+19) m_{1}(s)-s(16s^2+27s-5)]
\end{aligned}$
\item $\begin{aligned}[t]
k_s=& \frac{5m_{1^s}(s)}{32}[41m_4(s)+150m_{31}(s)+218m_{22}(s)+422 m_{211}(s)+816 m_{1111}(s)\\
& -2(75s+76) m_{3}(s)-2(211s+204) m_{21}(s)-48(17s+16) m_{111}(s)+\\
& +(211s^2+401s+140) m_{2}(s)+2(204s^2+377s+120)m_{11}(s)\\
& -2s(68s^2+185s+108) m_{1}(s)+s(s+2)(34s^2+53s-10)].
\end{aligned}$
\item $\begin{aligned}[t]
c_s=& \frac{m_{1^s}(s)}{64}[265m_4(s)+924m_{31}(s)+1330m_{22}(s)+2524 m_{211}(s)+4800 m_{1111}(s)\\
& -4(231s+190) m_{3}(s)-4(631s+510) m_{21}(s)-960(5s+4) m_{111}(s)+\\
& +2(631s^2+986s+280) m_{2}(s)+4(600s^2+929s+240)m_{11}(s)\\
& -4s(200s^2+449s+210) m_{1}(s)+s(200s^3+578s^2+363s-80)]
\end{aligned}$
\item $\begin{aligned}[t]
\chi'_s=& \frac{m_{1^s}(s)}{768}[675m_4(s)+2424m_{31}(s)+3510m_{22}(s)+6744 m_{211}(s)+12960 m_{1111}(s)\\
& -24(101s+95) m_{3}(s)-24(281s+255) m_{21}(s)-1440(9s+8) m_{111}(s)+\\
& +2(1686s^2+2991s+980) m_{2}(s)+24(270s^2+469s+140)m_{11}(s)\\
& -24s(90s^2+229s+125) m_{1}(s)+s(540s^3+1788s^2+1323s-280)].
\end{aligned}$            
\end{enumerate}
\end{lemma}
\begin{proof}
Immediate from Lemmas \ref{tf2} and \ref{gl1}.
\end{proof}
We now wish to compare all of the above functions. 

In order to do this, for any $b \in \Z$, define the polynomial $q_{s,b} \in \Q[x_1,\ldots,x_s]$ by 
$$q_{s,b} = bm_4(s)+10m_{22}(s)-10sm_2(s)+s(5s-b+5).$$
We have
\begin{lemma}
\label{gl4} 

For all $s \ge 4$ the following identities hold:
\begin{itemize}
\item[(1)] $g_{4,s}-f_{s,2,0}=\frac{m_{1^s}(s)}{4320}q_{s,8}.$
\item[(2)] $\chi'_s-f_{s,3,0}=\frac{m_{1^s}(s)}{3840}q_{s,9}.$
\end{itemize}
\end{lemma}
\begin{proof}
Immediate from Lemmas \ref{gl1} and \ref{gl2}. 
\end{proof}

We will also need the following crude estimate:

\begin{lemma}
\label{cg}
Let $b \in \{8, 9\}$ and let $s \ge 2$ be an integer. Then for all integers $d_i \ge 1, 1 \le i \le s$ with $\prod\limits_{i=1}^s d_i \ge 2$ we have that $q_{s,b}(d_1,\ldots, d_s) > 0$.
\end{lemma}
\begin{proof}
We proceed by induction on $s$. Since the case $s=2$ is easily verified, we assume that $s \ge 3$. Note now the identity:
$$q_{s+1,b}=q_{s,b}+bx_{s+1}^4+10x_{s+1}^2[m_2(s)-s-1]-10m_2(s)+10s-b+10.$$
Set 
$$r_b(t)=bt^4+10t^2[m_2(s)(d_1,\ldots,d_s)-s-1]-10m_2(s)(d_1,\ldots,d_s)+10s-b+10$$
so that
\begin{equation}
\label{gl3}
q_{s+1,b}(d_1,\ldots, d_{s+1}) = q_{s,b}(d_1,\ldots, d_s) + r_b(d_{s+1}).
\end{equation} 
There are two possible cases. 

The first one is when $\prod_{i=1}^s d_i \ge 2$, whence $q_{s,b}(d_1,\ldots,d_s)>0$ by induction hypothesis. Observe that in this case $(d/dt)r_b(t)>0$ when $t \ge 1$ and $r_b(1)=0$. It follows that $r_b(d_{s+1}) \ge 0$ for all $d_{s+1} \ge 1$, and consequently $q_{s+1,b}(d_1,\ldots, d_{s+1})>0$ by \eqref{gl3}. 

The second case is when $\prod_{i=1}^s d_i=1$, in which case $(d_1,\ldots, d_s)=(1,\ldots, 1)$. We have that $q_{s,b}(1, \ldots, 1)=0$ and $r_b(d_{s+1})=bd_{s+1}^4-10d_{s+1}^2-b+10>0$ since $d_{s+1}\ge 2$. Thus $q_{s+1,b}(d_1,\ldots, d_{s+1})>0$ by \eqref{gl3}, which completes the proof.
\end{proof}

\end{document}